%% file: cac-for-trees.tex
\documentclass[conference]{asl}
\usepackage{rotate}
\usepackage{url}

\def\confname{}
\def\copyrightyear{2000}

\newtheorem{theorem}{Theorem}[section]
\newtheorem{lemma}[theorem]{Lemma}

\newtheorem{corollary}[theorem]{Corollary}
\newtheorem{proposition}[theorem]{Proposition}

\theoremstyle{definition}
\newtheorem{definition}[theorem]{Definition}

\theoremstyle{remark}
\newtheorem{remark}[theorem]{Remark}
\newtheorem {question}[theorem]{Question}

\numberwithin{equation}{section}

\usepackage{hyperref}
\usepackage{cleveref}

\input{macros}

\usepackage{stmaryrd} 
\usepackage{xcolor} 
\usepackage{tikz-cd} 
\newcommand{\coloneqq}{\mathrel{\mathop:}=}

\newcommand{\ran}{\mathsf{ran}}
\usepackage{centernot} 

\newcommand{\lt}{{<}} 

\usepackage{xifthen}
\newcommand{\ce}{c.e.\ }
\newcommand{\tce}{c.e.}
\newcommand{\ie}{i.e.\ }
\newcommand{\TCAC}[1][]{\mathsf{CAC\ for\ \ifthenelse{\equal{#1}{}}{}{#1\ } trees}}
\newcommand{\SAC}{\mathsf{SAC}}
\newcommand{\CAC}{\mathsf{CAC}}
\newcommand{\TAC}{\mathsf{TAC}}

\newcommand{\SHER}{\mathsf{SHER}}
\newcommand{\SADS}{\mathsf{SADS}}

\newcommand{\HYP}{\mathsf{HYP}}
\newcommand{\PA}{\mathsf{PA}}
\newcommand{\DNC}{\mathsf{DNC}}

\newcommand{\BS}{\mathsf{B}\Sigma}

\title{The Reverse Mathematics of CAC for trees}

\author{Julien Cervelle}

\address{Univ Paris Est Creteil, LACL, F-94010 Creteil, France}

\email{julien.cervelle@u-pec.fr}

\urladdr{https://jc.lacl.fr}

\author{William Gaudelier}

\address{Univ Paris Est Creteil, LACL, F-94010 Creteil, France}

\email{william.gaudelier@gmail.com}

\author{Ludovic Patey}

\address{CNRS, Équipe de Logique\\Universit\'e de Paris\\ Paris, FRANCE}

\email{ludovic.patey@computability.fr}

\urladdr{http:/ludovicpatey.com}

\thanks{This project started as the study of Ramsey-like theorems for 3-variable forbidden patterns. The attempt to prove \Cref{p2-impl-rt22} naturally led to the study of the $\SHER$ principle, already defined by Dorais and al.~\cite{dorais2016uniform}. Thanks to multiple personal communications with François Dorais, we realized that the $\SHER$ principle is closely related to trees, and more precisely, equivalent to the Chain-Antichain principle for trees, a principle studied by Binns et al.\ in \cite{binns2014}. We later realized that $\SHER$ is also equivalent to~$\TAC+\BS^0_2$, where \(\TAC\) is an antichain principle for completely branching c.e. trees, defined by Conidis~\cite{conidistac}. Some of the results are therefore independent rediscoveries of some theorems from~\cite{binns2014,conidistac}, but in a more unified setting. The authors are thankful to Chris Conidis, François Dorais and Alberto Marcone for interesting comments and discussions. The authors are also thankful to the anonymous referee for his careful reading and his numerous improvement suggestions. The authors were partially supported by grant ANR ``ACTC'' \#ANR-19-CE48-0012-01.}

\begin{document}

\begin{abstract}
	\(\TCAC\) is the statement asserting that any infinite subtree of~\(\NN^{<\NN}\) has an infinite path or an infinite antichain. In this paper, we study the computational strength of this theorem from a reverse mathematical viewpoint. We prove that~\(\TCAC\) is robust, that is, there exist several characterizations, some of which already appear in the literature, namely, the statement~\(\SHER\) introduced by Dorais et al.\ \cite{dorais2016uniform}, and the statement~\(\TAC+\BS^0_2\) where \(\TAC\) is the tree antichain theorem introduced by Conidis \cite{conidistac}. We show that~\(\TCAC\) is computationally very weak, in that it admits probabilistic solutions.
\end{abstract}

\maketitle

\section{Introduction}


In this paper we study the computability-theoretic strength of the statement~\(\TCAC\), which is a variation on the well-studied chain-antichain theorem (\(\CAC\)). It turns out~\(\TCAC\) has different characterizations, making it a robust notion, suitable for future studies in reverse mathematics.


We are going to use two frameworks: reverse mathematics and computable reduction. For a good and more complete introduction to reverse mathematics, see Simpson \cite{simpson2009subsystems}, or Hirschfeldt \cite{hirschfeldt2015slicing} which also covers the computable reduction and classical results on Ramsey's theorem.

Reverse mathematics is a foundational program which seeks to determine the optimal axioms to prove ``ordinary" theorems. It itself uses the framework of subsystems of second-order arithmetic, with a base theory called~\(\RCA_0\), which informally captures ``computable mathematics".

Computable reduction makes precise the idea of being able to computably solve a problem~\(P\) using another problem~\(Q\). A problem is defined as a~\(\Pi_2^1\) formula in the language of second-order arithmetic, thought to be of the form~\(\forall X (\psi(X)\implies \exists Y \varphi(X, Y))\), where~\(\psi\) and~\(\varphi\) are arithmetical formulas. An instance of a problem is a set~\(X\) verifying~\(\psi(X)\), and a solution of an instance~\(X\) is a set~\(Y\) such that~\(\varphi(X, Y)\).
With this formalism, we say that ``\(P\) is computably reducible to~\(Q\)", and we write~\(P\leqslant_c Q\) if for any instance~\(I\) of~\(P\), there is an~\(I\)-computable instance~\(\widehat{I}\) of~\(Q\), such that, for any solution~\(\widehat{S}\) of~\(Q\), there is an~\(\widehat{S}\oplus I\)-computable solution~\(S\) of~\(P\). 


The early study of reverse mathematics has seen the emergence of four subsystems of second-order arithmetic, linearly ordered by the provability relation, such that most of the ordinary theorems are either provable in~\(\RCA_0\), or equivalent in~\(\RCA_0\) to one of them. These subsystems, together with~\(\RCA_0\), form the ``Big Five”. Among the theorems studied in reverse mathematics, Ramsey’s theorem for pairs~\(\RT_2^2\) plays an important role, since it is historically the first example of a natural statement which does not satisfy this empirical observation. The theorems we study in this paper are all consequences of~\(\RT_2^2\).

\subsection{A chain-antichain theorem for trees}

Among the consequences of Ramsey's theorem for pairs, the chain-antichain theorem received a particular focus in reverse mathematics.

\begin{definition}[\(\CAC\), chain-antichain theorem]
	Every infinite partial order has either an infinite chain or an infinite antichain.
\end{definition}

\(\CAC\) was first studied in \cite{hirschfeldt2007combinatorial} by Hirschfeldt and Shore, following a question raised by Cholak, Jockusch and Slaman in \cite[Question 13.8]{cholak2001strength} asking whether or not~\(\CAC\implies\RT_2^2\) over~\(\RCA_0\), for which they proved the answer is negative (Corollary 3.12).
The reciprocal~\(\RCA_0\vdash\RT_2^2\implies\CAC\) is easier to obtain, by defining a coloring such that~\(\{x, y\}\) has color 1 if its elements are comparable, and 0 otherwise. Any homogeneous set for this coloring is either a chain or an antichain, depending on its color.

In this article, we focus on the special case where the order is the predecessor relation~\(\prec\) on a tree.
\begin{definition}[\(\TCAC\)]
	Every infinite subtree of~\(\NN^{<\NN}\) has an infinite path or an infinite antichain.
\end{definition}

This statement was first introduced by Binns et al.\ in \cite{binns2014}, where the authors showed that every infinite computable tree must have either an infinite computable chain or an infinite~\(\Pi_1^0\) antichain. Furthermore, they showed that these bounds are optimal, by constructing an infinite computable tree which has no infinite~\(\Sigma_1^0\) chain or antichain. They also showed that~\(\RCA_0 + \WKL\centernot\vdash\TCAC[binary]\) (see \Cref{tcac2}).

\subsection{Ramsey-like theorems}

In \cite{patey2019ramsey}, Patey identified a formal class of theorems, encompassing several statements surrounding Ramsey's theorem. Indeed, many of them are of the form ``for every coloring~\(f:[\NN]^n\to k\) avoiding some set of forbidden patterns, there exists an infinite set~\(H\subseteq\NN\) avoiding some other set of forbidden patterns (relative to~\(f\))”

For example the Erd\H{o}s-Moser theorem (\(\EM\)) asserts that, ``for any coloring~\(f:[\NN]^2\to 2\), there exists an infinite set~\(H\subseteq\NN\) which is transitive for~\(f\)", \ie~\(\forall i<2, \forall x\lt y\lt z\in H, f(x, y)=i\land f(y, z)=i\implies f(x, z)=i\). In other terms, we want~\(H\) to avoid the patterns that would make it not transitive for \(f\), \ie~\(f(x, y)=i\land f(y, z)=i\land f(x, z)=1-i\) for any~\(i<2\).
Another example comes from~\(\ADS\) which is equivalent over~\(\RCA_0\) (see \cite[Theorem 5.3]{hirschfeldt2007combinatorial}) to the statement ``for any transitive coloring~\(f:[\NN]^2\to 2\) (\ie avoiding certain patterns), there exists an infinite set~\(H\subseteq\NN\) which is~\(f\)-homogeneous".
With these definitions, one sees that~\(\RT_2^2\) is equivalent to~\(\EM+\ADS\) over~\(\RCA_0\), since~\(\EM\) takes any coloring and ``turns it into" a transitive one, and~\(\ADS\) takes any transitive coloring and finds an infinite set which is homogeneous for it.

Forbidden patterns on 3 variables and 2 colors are generated by the following three basic patterns:
\begin{itemize}
	\item[(1)] \(f(x, y)=i\land f(y, z)=i\land f(x, z)=1-i\)
	\item[(2)] \(f(x, y)=i\land f(y, z)=1-i\land f(x, z)=i\)
	\item[(3)] \(f(x, y)=1-i\land f(y, z)=i\land f(x, z)=i\)
\end{itemize}
Avoiding them respectively leads to \textbf{transitivity}, \textbf{semi-ancestry}, and \textbf{semi-heredity} (for the color~\(i\)). Each of them generates two ramsey-like statements, one restricting the input coloring, and one restricting the output infinite set, namely ``for any 2-coloring of pairs avoiding the forbidden pattern, there exists an infinite homogeneous set'' and ``for any 2-coloring of pairs, there exists an infinite set~\(H\subseteq\NN\) which avoids the forbidden pattern''. We now survey the known results about these three patterns.

\emph{Transitivity.} The statement ``for any 2-coloring of pairs, there exists an infinite set which is transitive for some color'' is a weaker version of~\(\EM\). The Erdös-Moser theorem was proven to be strictly weaker than Ramsey's theorem for pairs over~\(\RCA_0\) by Lerman, Solomon and Towsner~\cite[Corollary 1.16]{lerman2013separating}. On the other hand, the statement ``for any 2-coloring of pairs which is transitive for some color, there exists an infinite homogeneous set'' is equivalent to~\(\CAC\) (see \cite[Theorem 5.2]{hirschfeldt2007combinatorial}), which is also known to be strictly weaker than~\(\RT_2^2\) over~\(\RCA_0\) (see Hirschfeldt and Shore~\cite[Corollary 3.12]{hirschfeldt2007combinatorial}).

\emph{Semi-ancestry.} 
The statement ``for any 2-coloring of pairs which has semi-ancestry for some color, there exists an infinite homogeneous set'' is a consequence of the statement~\(\mathsf{STRIV}\), defined by Dorais et al.~\cite[Statement 5.12]{dorais2016uniform}), because a 2-coloring is semi-trivial if and only if it has semi-ancestry.
And~\(\mathsf{STRIV}\) itself is equivalent to~\(\forall k, \RT_k^1\) (see the remark below its definition).
The statement ``for any 2-coloring of pairs, there exists an infinite set which has semi-ancestry for some color'' is equivalent to~\(\RT_2^2\) (see \Cref{p1-impl-rt22}).

\emph{Semi-heredity.} 
The statement ``for any 2-coloring of pairs which is semi-hereditary for some color, there exists an infinite homogeneous set'' is the statement~\(\SHER\), which was first introduced by Dorais et al.~\cite[Statement 5.11]{dorais2016uniform}. In \Cref{sect:sher}, we will show that it is equivalent to~\(\TCAC\).
Finally, the statement ``for any 2-coloring of pairs, there exists an infinite set which is semi-hereditary for some color'' is equivalent to~\(\RT_2^2\) (see \Cref{p2-impl-rt22}).

\subsection{Notation}
Let the symbols~\(\exists^\infty x\) and~\(\forall^\infty x\) be abbreviations for \(\forall y, \exists x>y\) and~\(\exists y, \forall x>y\) respectively. In particular they are the dual of each other.

Given~\(x, y\in \NN\cup\{\pm\infty\}\), we define~\(\llbracket x, y \rrbracket\coloneqq \{z\in\NN\mid z\geqslant x\land z\leqslant y\}\), an inequality being strict when its respective bracket is flipped, e.g.\ \(\llbracket x, y \llbracket\ \coloneqq  \{z\in\NN\mid z\geqslant x\land z<y\}\). As in set theory, an integer~\(n\) can also represent the set of all integers strictly smaller to it, \ie~\(\llbracket 0, n\llbracket\). Moreover \(\langle -, \ldots, - \rangle\) represents a bijection from \(\NN^n\) to \(\NN\) (for some \(n\)), which verifies \(\forall x, y\in\NN, \langle x_1,\ldots, x_k\rangle\geqslant \max\{x_1,\ldots, x_k\}\) and which is increasing on each variable. Given a set \(A\subseteq\NN\), we denote by \(\chi_A\) its \textbf{indicator function}.

A \textbf{string} is a finite sequence of integers, and the set of all strings is denoted~\(\NN^{<\NN}\). In particular, a \textbf{binary string} is a finite sequence of 0 and 1, and the set of all binary strings is denoted~\(2^{<\NN}\). The \textbf{length} of a given string~\(\sigma:n\to\NN\) is the integer~\(|\sigma|\coloneqq n\). The \textbf{empty string}~\(\varnothing\to\NN\) is denoted~\(\varepsilon\). Given two strings~\(\sigma\) and~\(\tau\) of respective length~\(\ell\) and~\(m\), we define their \textbf{concatenation} as the finite sequence~\(\sigma\cdot\tau:\ell+m\to\NN\) which, given~\(j<\ell+m\), associates~\(\sigma(j)\) if~\(j<\ell\), and~\(\tau(j-\ell)\) otherwise.
So we usually write a string~\(\sigma:n\to\NN\) as~\(\sigma_0\cdot\ldots\cdot \sigma_{n-1}\), where~\(\forall j<n, \sigma_j\coloneqq \sigma(j)\).
We define the partial order on strings ``is \textbf{prefix} of", denoted \(\prec\),  by~\(\sigma\prec\tau\iff |\sigma|<|\tau|\land\forall j<|\sigma|, \sigma(j)=\tau(j)\), and denote by~\(\preccurlyeq\) its reflexive closure. When two strings \(\sigma\) and \(\tau\) are \textbf{incomparable} for \(\prec\), we write \(\sigma\bot\tau\).

A \textbf{tree}~\(T\) is a subset of~\(\NN^{<\NN}\) which is downward closed for~\(\prec\), i.e. \(\forall \sigma\in T, \tau\prec\sigma\implies \tau\in T\). A \textbf{binary tree} is a subset of~\(2^{<\NN}\) which is downward closed for~\(\prec\). A subset~\(S\subseteq T\) of a tree is a \textbf{chain} when it is linearly ordered for~\(\prec\), and it is an \textbf{antichain} when its elements are pairwise incomparable for~\(\prec\).

\subsection{Organization of the paper}

In \Cref{sect:robustness}, we prove the robustness of~\(\TCAC\) in reverse mathematics and over the computable reduction, by proving its equivalence with several variants of the statement. In \Cref{sect:probabilistic-proof}, we provide a probabilistic proof of~\(\TCAC\), and give a precise analysis of this proof in terms of \(\DNC\) functions. 
In \Cref{sect:adsem} we show that both~\(\ADS\) and~\(\EM\) imply~\(\TCAC\). In \Cref{sect:hyper} we prove there is a computable instance of~\(\TAC\) whose solutions are all of hyperimmune degree, with almost explicit witness. In particular we show a computable instance of~\(\TAC\) can avoid a uniform sequence of~\(\Delta_2^0\) sets.
In \Cref{sect:sher}, we show the equivalence between~\(\TCAC\) and~\(\SHER\).
In \Cref{sect:stable}, we explore the relations between stable versions of previously mentioned statements, namely~\(\TCAC\), \(\ADS\) and~\(\SHER\), leading to the result that~\(\TCAC[stable\ \tce]\) admits low solutions.
Finally, in \Cref{sect:conclu}, we conclude the paper with remaining open questions and summary diagrams. 

\section{Equivalent definitions}\label[section]{sect:robustness}
In this section, we study some variations of~\(\TCAC\), and prove they are all equivalent. We also study~\(\TAC\) and show that \(\TAC+\BS^0_2\) is equivalent to \(\TCAC\). We start by defining these statements.

\begin{definition}[\textsf{CAC for \ce (binary) trees}]\label[definition]{tcac2}
	Every infinite \ce (binary) subtree of~\(\NN^{<\NN}\) has an infinite path or an infinite antichain.
\end{definition}

\begin{remark}
	In the context of reverse mathematics ``being \ce\kern-.4em" is a notion that is relative to the model considered, \ie an object is \ce when it can be approximated in a \ce manner by objects from the model.
\end{remark}

\begin{definition}[Completely branching tree]
	A node~\(\sigma\) of a tree~\(T\subseteq \NN^{<\NN}\) is a \textbf{split node} when there is \(n_0, n_1\in\NN\) such that \(\forall i<2, \sigma\cdot n_i\in T\).
In particular, if \(T\) is a binary tree, then \(\sigma\) is a split node when both~\(\sigma\cdot 0\in T\) and~\(\sigma\cdot 1\in T\).
	A tree~\(T\subseteq 2^{<\NN}\) is \textbf{completely branching} when, for any of its node~\(\sigma\), if~\(\sigma\) is not a leaf then it is a split node.
\end{definition}

The following statement was introduced by Conidis~\cite{conidistac} (personal communication), motivated by the reverse mathematics of commutative noetherian rings. 

\begin{definition}[\(\TAC\) \cite{conidistac}, tree antichain theorem]
	Any infinite \ce subtree of~\(2^{<\NN}\) which is completely branching, contains an infinite antichain.
\end{definition}

Conidis proved that $\TAC$ follows from~$\CAC$ over~$\RCA_0$, and constructed an instance of~$\TAC$, all of whose solutions are of hyperimmune degree. In particular, $\RCA_0 + \WKL \nvdash \TAC$. 
Now we can proceed with the proof of the equivalence.

\begin{theorem}\label[theorem]{big-equiv}
The following statements are equivalent over~$\RCA_0$ and computable reduction:
\begin{itemize}
	\item[(1)] \(\TCAC\)
	\item[(2)] \(\TCAC[\tce]\)
	\item[(3)] \(\TCAC[\ce binary]\)
	\item[(4)] \(\TAC+\BS^0_2\)
\end{itemize}
\end{theorem}
\begin{proof}
\((2) \implies (1)\) and~\((2) \implies (3)\) are immediate.
\((3) \implies (4)\) is \Cref{tcac-ce-binary-implies-tac} and \Cref{tcac-ce-bin-impl-bs2}.
\((4) \implies (2)\) is \Cref{tac-imples-tcac-ce}.
\((1) \implies (2)\) is \Cref{implyce}.
\end{proof}

We shall see in \Cref{sect:probabilistic-proof} that the use of~$\BSig_2$ is necessary in the equivalence above, as $\TAC$ does not imply $\BSig_2$ over~$\RCA_0$.

\begin{proposition}\label[proposition]{tcac-ce-binary-implies-tac}
	\(\RCA_0\vdash \TCAC[\ce binary]\implies\TAC\) and~\(\TAC\leqslant_c\TCAC[\ce binary]\).
\end{proposition}
\begin{proof}
	Let~\(T\subseteq 2^{<\NN}\) be an infinite completely branching \ce tree. By the statement \(\TCAC[\ce binary]\), either there is an infinite antichain, or an infinite path~$P$. In the former case, we are done. In the latter case, using the fact that~$T$ is completely branching, the set~\(\{\sigma\cdot(1-i)\mid \sigma\cdot i\prec P\}\) is an infinite antichain of~\(T\).
\end{proof}

\begin{proposition}\label[proposition]{tcac-ce-bin-impl-bs2}
	\(\RCA_0\vdash \TCAC[\ce binary]\implies\forall k, \RT^1_k\)\\and~\(\forall k, \RT^1_k\leqslant_c\TCAC[\ce binary]\).
\end{proposition}
\begin{proof}
	Let \(f:\NN\to k\) be a coloring, there are two possibilities. Either \(\exists i<k, \exists^\infty x, f(x)=i\), in which case there is an infinite computable \(f\)-homogeneous set. Otherwise \(\forall i<k, \forall^\infty x, f(x)\neq i\), in which case we define an infinite binary \(\tce\) tree \(T\), via a strictly increasing sequence of computable trees \((T_j)_{j\in\NN}\) defined by \(T_0 \coloneqq \{0^i\mid i<k\}\) and \(T_{s+1} \coloneqq T_s\cup\{0^{f(s)}\cdot 1^{m+1}\}\), where \(m\) is the number of \(x<s\) such that \(f(x)=f(s)\).
	
	Every antichain in \(T\) is of size at most \(k\), thus, by \(\TCAC[\ce binary]\), \(T\) must contain an infinite path, and so \(\exists i<k, \exists^\infty x, f(x)=i\), which is a contradiction.
\end{proof}

\begin{proposition}\label[proposition]{tac-imples-tcac-ce}
	\(\RCA_0\vdash\TAC+\BS^0_2\implies\TCAC[\tce]\) and
	
	\(\TCAC[\tce] \leqslant_c \TAC\)
\end{proposition}
\begin{proof}
	Let~\(T\subseteq\NN^{<\NN}\) be an infinite \ce tree. We can deal with two cases directly: when~\(T\) has a node with infinitely many immediate children, as they contain a computable infinite antichain; and when~\(T\) has finitely many split nodes, in which case it has finitely many paths \(P_0, \ldots, P_{k-1}\), which are all computable. Moreover one of them is infinite, as otherwise they would all be finite, \ie \(\forall i<k, \exists s, \forall n>s, n\notin P_i\), and thus their union would be finite since \(\BPi_1\) (which is equivalent to \(\BS^0_2\)) yields \(\exists b, \forall i<k, \exists s<b, \forall n>s, n\notin P_i\). But since \(T=\bigcup_{i<k} P_i\), this would lead to a contradiction.

	A \textbf{split triple} of~$T$ is a triple~\((\mu, n_0, n_1)\in T\times\NN\times\NN\) such that $\mu, \mu\cdot n_0, \mu\cdot n_1 \in T$. In particular, $\mu$ is a split node in~$T$.

%
%
%
%
%
	
	\textbf{Idea}. The general idea is to build greedily a completely branching \ce tree~$S$ by looking for split triples in~$T$, and mapping them to split nodes in~$S$. This correspondence is witnessed by an injective function $f : S \to T$ that will be constructed alongside \(S\).
	The main difficulty is that, since~\(T\) is c.e., a split node~\(\rho\) can be discovered after~\(\mu\) even though~\(\rho\prec\mu\), which means that we will not be able to ensure that~$S$ can be embedded in~$T$. In particular,~$f$ will not be a tree morphism.
	However, the only property that needs to be ensured is that for every infinite antichain~\(A\) of~\(S\), the set~\(f(A)\) will be an infinite antichain of~\(T\). To guarantee this, the function~\(f\) needs to verify
	\begin{equation}\label{eqn:ac}
		\forall \sigma, \nu\in S, \sigma\bot\nu \implies f(\sigma)\bot f(\nu)\tag{\textasteriskcentered}
	\end{equation}
	During the construction, we are going to associate to each node~$\sigma \in S$ a set~$N_\sigma \subseteq T$, which might decrease in size over time ($N^0_\sigma \supseteq N^1_\sigma \supseteq \dots$), with the property that at every step~$s$, the elements of~$\{ N^s_\sigma : \sigma \in S\}$ are pairwise disjoint, and their union contains cofinitely many elements of~$T$. 
	The role of~$N_\sigma$ is to indicate that ``if a split triple is found in~\(N_{\sigma}\), then the nodes in~\(S\), associated via~\(f\), must be above~\(\sigma\)''. 
%
%
%
	
	\textbf{Construction.} 
	Initially, $N^0_\varepsilon \coloneqq  T$, $S \coloneqq  \{\varepsilon\}$ and $f(\varepsilon) \coloneqq  \varepsilon$.
	
	At step~$s$, suppose we have defined a finite, completely branching binary tree~$S \subseteq 2^{<\NN}$,
		and for every~$\sigma \in S$, a set~$N^s_\sigma \subseteq T$ such that $\{N^s_\sigma : \sigma \in S\}$ forms a partition of \(T\) minus finitely many elements. Moreover, assume we have defined a mapping~$f : S \to T$.
	
	Search for a split triple $(\mu, n_0, n_1)$ in $\bigcup_{\sigma \in S} N^s_\sigma$. Let~$\sigma \in S$ be such that $\mu \in N^s_\sigma$. Let~$\tau$ be any leaf of~$S$ such that $\tau\succeq\sigma$ (for example pick the left-most successor of \(\sigma\)).
	Add~$\tau \cdot 0$ and $\tau \cdot 1$ to $S$, and set~$f(\tau \cdot i) = \mu \cdot n_i$ for each~$i<2$. Note that~$S$ is still completely branching.
	
	Then, split~$N^s_\sigma$ into three disjoint subsets~$N^{s+1}_\sigma$, $N^{s+1}_{\tau \cdot 0}$, $N^{s+1}_{\tau \cdot 1}$ as follows:
	(for~\(i<2\)) \(N^{s+1}_{\tau \cdot i}\coloneqq \{\rho\in N^s_\sigma \mid \rho\succcurlyeq \mu \cdot n_i\}\) and~\(N^{s+1}_\sigma\coloneqq \{\rho\in N^s_\tau \mid \forall i<2, \rho\bot\mu\cdot n_i\}\).
	Note that these sets do not form a partition of~\(N^s_\sigma\) as we missed the nodes in~\(\{\rho\in N^s_\sigma \mid \rho \preccurlyeq \mu\}\), fortunately there are only finitely many of them.
	Lastly, set~$N^{s+1}_\nu\coloneqq N^s_\nu$ for every~$\nu\in S - \{\sigma, \tau \cdot 0, \tau \cdot 1 \}$.
	
	\textbf{Verification.}
	First, let us prove that at any step $s$, $\bigcup_{\sigma \in S} N^s_\sigma$ contains infinitely many split triples, which ensures that the search always terminates. Note that $T-\bigcup_{\sigma\in S} N^s_\sigma$ is a \ce subset of $\bigcup_{\sigma \in S} \{\rho\mid \rho \preceq \sigma \}$, hence exists by bounded $\Sigma_1$ comprehension, which follows from $\RCA_0$. Moreover, by assumption, $T$ is a finitely branching \ce tree, which means that every~$\rho \in T-\bigcup_{\sigma\in S} N^s_\sigma$ belongs to a bounded number of split triples of~$T$. By $\BSig_1$ (which follows from $\RCA_0$), the number of split triples in~$T$ which involve a node from $T-\bigcup_{\sigma\in S} N^s_\sigma$ is bounded. Since by assumption, $T$ contains infinitely many split triples, $T-\bigcup_{\sigma\in S} N^s_\sigma$  must contain infinitely many of them.	
	
	Second, we prove by induction on \(s\) that \(\forall s\in \NN, \forall \sigma\neq\nu\in S, N_\sigma^s\bot N_\nu^s\), \ie \(\forall \mu\in N_\sigma^s, \forall \rho\in N_\nu^s, \mu\bot\rho\). At step 0, the assertion is trivially verified.
	At step \(s\), suppose we found the split triple \((\mu, n_0, n_1)\) in the set \(N_\sigma^s\), and that \(\forall i<2, f(\tau\cdot i)=\mu\cdot n_i\) where \(\tau\succcurlyeq\sigma\). Since \(\mu\) was found in \(N_\sigma^s\), the latter is split into \(N_{\tau\cdot 0}^{s+1}\), \(N_{\tau\cdot 1}^{s+1}\) and \(N_\sigma^{s+1}\), the other sets remain identical. By construction, and because they are all subsets of \(N_\sigma^s\), the assertion holds.
	
	We now prove \eqref{eqn:ac}, consider \(\sigma, \nu\in S\) such that \(\sigma\bot\nu\).
	WLOG suppose \(\nu\) was added to \(S\) sooner than \(\sigma\), more precisely \(f(\nu)\) appeared (as child in a split triple) at step \(s\) in some set \(N_{-}^s\), so \(N_\nu^{s+1}\) contains \(f(\nu)\) by construction. Since \(\sigma\) was added to \(S\) after \(\nu\), there exists \(\rho\in S\) such that \(f(\sigma)\in N_\rho^{s+1}\). By contradiction \(\rho\neq\sigma\) holds, as otherwise \(f(\sigma)\in N_\tau^{s+1}\), and so \(\sigma\) would extend \(\tau\) by construction of \(S\). Thus by using the previous assertion, we deduce \(f(\sigma)\bot f(\nu)\).
\end{proof}

\begin{proposition}\label[proposition]{implyce}
	\(\RCA_0\vdash\TCAC\implies\TCAC[\tce]\) and
	
	\(\TCAC[\tce]\leqslant_c\TCAC\).
\end{proposition}
\begin{proof}
	Let~\(T\subseteq\NN^{<\NN}\) be a \ce tree. We define the computable tree~\(S\subseteq\NN^{<\NN}\) by
	\(\langle n_0, s_0\rangle\cdot\ldots\cdot\langle n_{k-1}, s_{k-1}\rangle\in S\) if and only if for all~\(j<k\), \(s_j\) is the smallest integer such that~\(n_0\cdot\ldots\cdot n_j\in T[s_j]\), where $T[s_j]$ is the approximation of~$T$ at stage~$s_j$.
	
	By~\(\TCAC\), there is an infinite chain (resp. antichain) in~\(S\), and by forgetting the second component of each string, we obtain an infinite chain (resp. antichain) in~\(T\).
\end{proof}

Before finishing this section, we introduce a set version of the principle~$\TAC$, which is more convenient to manipulate than $\TAC$. Indeed, when working with $\TAC$, the downward-closure of the tree is not relevant, and we naturally end up taking an infinite computable subset of the tree rather than working with the c.e.\ tree. This motivates the following definitions.

\begin{definition}
A set~$X \subseteq 2^{<\NN}$ is \textbf{completely branching} if for every~$\sigma \in 2^{<\NN}$, $\sigma\cdot 0 \in X$ iff $\sigma\cdot 1 \in X$. 
\end{definition}

Note that the above definition is compatible with the notion of completely branching tree.

\begin{definition}[$\SAC$, set antichain theorem]
Every infinite completely branching set~$X \subseteq 2^{<\NN}$ has an infinite antichain.
\end{definition}

The set antichain theorem is equivalent to the tree antichain theorem, as shows the following lemma.

\begin{lemma}\label[lemma]{sac-equiv-tac}
$\RCA_0 \vdash \SAC \iff \TAC$.
\end{lemma}
\begin{proof}
$\SAC \implies \TAC$. Let~$T \subseteq 2^{<\NN}$ be an infinite, completely branching \ce tree. Let~$S \subseteq T$ be an infinite computable, completely branching set. By $\SAC$, there is an infinite antichain~$A \subseteq X$. In particular,~$A$ is an antichain for~$T$.

$\SAC \impliedby \TAC$. Let~$S \subseteq 2^{<\NN}$ be an infinite computable, completely branching set.
One can define an infinite computable tree~$T \subseteq \NN^{<\NN}$ by letting~$\sigma \in T$ iff for every $n < |\sigma|$, $\sigma(n)$ codes for a binary string $\tau_n \in S$, such that for every~$n < |\sigma|-1$, $\tau_{n} \prec \tau_{n+1}$, and there is no string in~\(S\) strictly between $\tau_n$ and $\tau_{n+1}$. The tree $T$ is such that every chain in $T$ codes for a chain in~$S$, and every antichain in $T$ codes for an antichain in~$S$.
We can see~$T$ as an instance of $\TCAC$.
Moreover, since $S$ is completely branching, then $T$ has infinitely many split triples, so the proof of \Cref{tac-imples-tcac-ce} applied to this instance~$T$ of $\TCAC$ does not use $\BSig_2$. Thus there is either an infinite chain for~$T$, or an infinite antichain for~$S$. With the appropriate decoding, we obtain an infinite antichain for~$S$.
\end{proof}

\section{Probabilistic proofs of~\(\SAC\)}\label[section]{sect:probabilistic-proof}
The restriction of $\CAC$ to trees yields a strictly weaker statement from the viewpoint the arithmetical bounds in the arithmetic hierarchy. Indeed, by Herrmann~\cite[Theorem 3.1]{herrmann2001infinite}, there is a computable partial order with no $\Delta^0_2$ infinite chain or antichain, while by Binns et al.\ in \cite[Theorem 6.2]{binns2014}, every infinite computable tree must have either an infinite computable chain or an infinite~\(\Pi_1^0\) antichain. 
In this section, we go one step further in the study of the weakness of~\(\TCAC\) by proving that~$\SAC$ admits probabilistic solutions. Before this, we prove two technical lemmas.

\begin{lemma}[$\RCA_0$]\label[lemma]{lem:completely-branching-finite-antichain-exists}
Let~$S \subseteq 2^{<\NN}$ be an infinite completely branching set.
Then for every~$n$, there exists an antichain of size~$n$.
\end{lemma}
\begin{proof}
By finite Ramsey's theorem for pairs and 2 colors (which holds in~$\RCA_0$),
there exists some~$p \in \NN$ such that for every 2-coloring of~$[p]^2$,
there exists a homogeneous set of size~$n$.
Since~$S$ is infinite, there exists a subset $P \subseteq S$ of size~$p$.
By choice of~$p$, there exists a subset~$Q \subseteq P$ of size~$n$ such that
either~$Q$ is a chain, or an antichain. In the latter case, we are done.
In the former case, since $S$ is completely branching, the set~$\{\hat{\sigma}\mid \sigma \in Q \} \subseteq S$ is an antichain, where~$\hat{\sigma}$ is the string obtained from~$\sigma$ by flipping its last bit.
\end{proof}

\begin{lemma}[$\RCA_0$]\label[lemma]{lem:completely-branching-at-most-one-wrong}
Let~$S \subseteq 2^{<\NN}$ be an infinite completely branching set.
Then for every antichain~$A \subseteq S$, for all but at most one~$\sigma \in A$,
the set~$S_\sigma \coloneqq \{ \tau \in S\mid \sigma \bot \tau \mbox{ and } |\tau| > |\sigma| \}$ is infinite and completely branching.
\end{lemma}
\begin{proof}
First, since~$S \subseteq 2^{<\NN}$ completely branching, then for every~$\sigma \in 2^{<\NN}$, the set~$S_\sigma$ is completely branching. Suppose for the sake of contradiction that there exists two strings
$\sigma, \rho \in A$ such that $S_\sigma$ and $S_\rho$ are both finite. Then pick any~$\tau \in S - (S_\sigma \cup S_\rho)$ with $|\tau| > \max(|\sigma|, |\rho|)$. It follows that $\sigma \prec \tau$ and $\rho \prec \tau$, and thus that~$\sigma$ and~$\rho$ are comparable, contradicting the fact that~$A$ is an antichain.
\end{proof}

\begin{proposition}\label[proposition]{proba}
	The measure of the oracles computing a solution for a computable instance of~\(\SAC\) is 1. 
\end{proposition}
\begin{remark}
	The proof of the above proposition is carried out purely as a computability statement, hence we have access to as much induction as needed. 
\end{remark}
\begin{proof}
	Let~\(S \subseteq 2^{<\NN}\) be an infinite computably branching set. We are going to build a decreasing sequence of infinite completely branching sets of strings~$S_0 \supseteq S_1 \supseteq \dots$, with~$S_0\coloneqq S$, together with finite antichains \(A_i\subseteq S_i\) (for \(i\in\NN\)), in order to have an infinite antichain \(A\coloneqq\{\sigma_i\mid i\in\NN\}\) where \(\sigma_i\in A_i\).
	
	This construction will work with positive probability, and since the class of oracles computing a solution to the instance \(S\) is invariant under Turing equivalence, this implies that this class has measure 1. Indeed, by Kolmogorov's 0-1 law, every measurable Turing-invariant class has either measure 0 or 1.

	First, let $S_0\coloneqq S$.
	At step~\(k\), assume the sets~\(S_0 \supseteq S_1\supseteq\ldots\supseteq S_k\) and \(A_0, \ldots, A_{k-1}\) have been defined, as well as the finite antichain \(\{\sigma_0, \ldots, \sigma_{k-1}\}\), such that \(\forall \tau\in S_k, \forall i<k, \sigma_i\bot\tau\).
	
	Search computably for a finite antichain~\(A_k\subseteq S_k\) of size~\(2^{k+2}\). If found, pick an element~\(\sigma_k\in A_k\) at random. Then define~\(S_{k+1}\coloneqq \{\tau\in S_k\mid \sigma_k\bot \tau \mbox{ and } |\tau| > |\sigma_k| \}\) for the next step.
	
	If the procedure never stops, it yields an infinite antichain~\(A\coloneqq \{\sigma_i\mid i\in\NN\}\) thanks to the definition of the sets \((S_i)_{i<k}\). Assuming that $S_k$ is an infinite completely branching set, \Cref{lem:completely-branching-finite-antichain-exists} ensures that $A_k$ will be found. 
	
	However, if at any point, $S_k$ is not an infinite completely branching set, then at some point~\(t\) we will not be able to find a large enough~\(A_t\) in it.
	If this happens, since~$S_{k+1}$ is completely determined by $S_k$ and $\sigma_k$, it means that we have chosen some ``bad" $\sigma_k \in A_k$. Luckily, by \Cref{lem:completely-branching-at-most-one-wrong}, there is at most one element of this kind in~$A_k$. Thus, if we pick $\sigma_k$ at random in~$A_k$, we have at most $\frac{1}{|A_k|}=\frac{1}{2^{k+2}}$ chances for this case to happen. So the overall probability that this procedure fails is less than~\(\sum_{k\geqslant0} \frac{1}{2^{k+2}}=\frac{1}{2}\). Hence we found an antichain with positive probability.
\end{proof}

Very few theorems studied in reverse mathematics admit a probabilistic proof. \Cref{proba} provides a powerful method for separating the statement $\TCAC$ from many theorems in reverse mathematics. In what follows, $\mathsf{AMT}$ stands for the Atomic Model Theorem, studied by Hirschfeldt, Shore and Slaman~\cite{hirschfeldt2009atomic},
$\mathsf{COH}$ is the cohesiveness principle, defined by Cholak, Jockusch and Slaman~\cite[Statement 7.7]{cholak2001strength},
and $\mathsf{RWKL}$ is the Ramsey-type Weak K\"onig's lemma, defined by Flood~\cite[Statement 2]{flood2012reverse} under the name~$\mathsf{RKL}$.

\begin{corollary}
Over $\RCA_0$, $\TCAC$ implies none of $\mathsf{AMT}$, $\COH$ and $\mathsf{RWKL}$.
\end{corollary}
\begin{proof}
	These three statements have a computable instance such that the measure of the oracles computing a solution is 0, see Astor et al.\ \cite{astor20xxweakness}.
\end{proof}

The argument of \Cref{proba} can be formalized over~$\RCA_0$ to yield the following result.

\begin{definition}[$2\mbox{-}\mathsf{RAN}$]
For every sequence of uniformly $\Pi^0_2$ binary trees~$T_0, T_1,\ldots$
such that, for every~$n$, $\mu([T_n]) > 1-2^{-n}$, there is some~$n$ and some set~$X$ such that
$X \in [T_n]$.
\end{definition}

\begin{proposition}
$\RCA_0 \vdash 2\mbox{-}\mathsf{RAN} \implies \SAC$.
\end{proposition}
\begin{proof}
For every~$n$, consider the construction of \Cref{proba}, where the antichain~$A_k$ is of size~$2^{n+k+1}$ instead of~$2^{k+2}$. For each~$k$, let~$\sigma_k \in A_k$ be the unique ``bad" choice (if it exists), that is, which makes the set~$S_{k+1}$ finite, and let~$\tau_k$ be the string of length~$n+k+1$ corresponding to the binary representation of the rank of~$\sigma_k$ in~$A_k$ for some fixed order on binary strings. Then one can compute \(\sigma_k\) from  \(\tau_k\) and the finite set $A_k$. Note that \(\tau_k\) is undefined when \(\sigma_k\) does not exist.

Consider the $\Sigma^0_2$ class~$\U_n \coloneqq \{ X \in 2^\NN\mid \exists k, \tau_k \prec X \}=\bigcup_k [\tau_k]$. It verifies $$\mu(\U_n) \leqslant \sum_{{k \geqslant 0}\atop {\sigma_k \text{exists}}} \mu\big([\tau_k]\big) \leqslant \sum_{k \geqslant 0} \frac{1}{2^{n+k+1}} = 2^{-n}$$
Let~$T_n$ be a $\Pi^0_2$ tree such that $[T_n] = 2^\NN - \U_n$. We can now consider the sequence of trees \((T_n)_{n\in\NN}\).
By $2\mbox{-}\mathsf{RAN}$, there is some~$n$ and some~$X \in [T_n]$. For any instance of \(\SAC\), find a solution by running the construction given in \Cref{proba} with the help of~$X$ to avoid the potential ``bad" choice in each~$A_k$.
\end{proof}

\begin{corollary}
Over $\RCA_0$, $\SAC$ (and therefore $\TAC$) implies none of~$\BSig_2$ and $\TCAC$.
\end{corollary}
\begin{proof}
Slaman [unpublished] proved that $2\mbox{-}\mathsf{RAN}$ does not imply~$\BSig_2$ over~$\RCA_0$; The result can also be found in~\cite{phpbs2}. The corollary follows from $\RCA_0 \vdash \SAC \implies \TAC$ (\Cref{sac-equiv-tac}) and $\RCA_0 \vdash \TAC+\BS^0_2 \iff \TCAC$ (\Cref{big-equiv}).
\end{proof}

We are now going to refine \Cref{proba} by proving that some variant of~\(\DNC\) is sufficient to compute a solution of~\(\SAC\).

\begin{definition}[Diagonally non-computable function]
A function $f : \NN \to \NN$ is \textbf{diagonally non-computable relative to~$X$} (or $\DNC(X)$) if for every~$e$, $f(e) \neq \Phi^X_e(e)$. Whenever~$f$ is dominated by a function $h : \NN \to \NN$, then we say that $f$ is $\DNC_h(X)$.
A Turing degree is $\DNC_h(X)$ if it contains a $\DNC_h(X)$ function.
\end{definition}

The following lemma gives a much more convenient way to work with $\DNC_h(X)$ functions.

\begin{lemma}[Folklore]\label[lemma]{dncequiv}
	Let~\(A, X\) be subsets of \(\NN\). The following are equivalent:
	\begin{itemize}
		\item[(1)] \(A\) is of degree~\(\DNC_h(X)\) for some computable (primitive recursive) function~\(h:\NN\to\NN\).
		\item[(2)] \(A\) computes a function~\(g:\NN^2\to\NN\) such that
		$$\forall e, n, |W_e^X|\leqslant n\implies g(e, n)\notin W_e^X$$ and which is dominated by a computable function~\(b:\NN^2\to\NN\), \ie
		$$\forall e, n, g(e, n)<b(e, n)$$
	\end{itemize}
\end{lemma}
\begin{proof}
	\((2)\implies(1)\).
	Let~\(i:\NN\to\NN\) be a computable (primitive recursive) function such that for any~\(e\in\NN\) and~\(B\subseteq\NN\) we have~\(\Phi^B_{i(e)}(x)\downarrow\iff x=\Phi_e^B(e)\).
	Thus \[W_{i(e)}^B=\begin{cases}
		\{\Phi_e^B(e)\}&\text{if }e\in B'\\
		\varnothing&\text{otherwise}
	\end{cases}\]

	From there, define the~\(A\)-computable function~\(f:\NN\to\NN\) by~\(f:e\mapsto g(i(e), 1)\). It is~\(\DNC(X)\) because~\(g(i(e), 1)\notin W_{i(e)}^X\) since~\(|W_{i(e)}^X|\leqslant 1\).
	Moreover, $f$ is dominated by the computable function~\(e\mapsto b(i(e), 1)\), because~\(b\) computably dominates~\(g\).

	\((1)\implies(2)\).
	Let~\(f\) be a~\(\DNC_h(X)\) function computed by~\(A\). Given the pair~\(e, n\), we describe the process that defines~\(g(e, n)\).
	
	\textbf{Construction.}
	For each~\(i<n\), we compute the code~\(u(e, i)\) of the~\(X\)-computable function which, on any input, looks for the~\(i^{\text{th}}\) element of~\(W_e^{X}\).
	If it finds such an element, then it interprets it as an~\(n\)-tuple~\(\langle k_0, \ldots, k_{n-1}\rangle\) and returns the value~\(k_i\).
	If it never finds such an element, then the function diverges.
	Finally we define~\(g:e, n\mapsto \langle f(u(e, 0)), \ldots, f(u(e, n-1))\rangle\)
	
	\textbf{Verification.}
	First, since~\(f\) is dominated by~\(h\), and since the function~\(\langle -,\ldots, -\rangle\) computing an~\(n\)-tuple is increasing on each variable, we can dominate~\(g\) with the computable function
	$$b:e, n\mapsto \langle h(u(e, 0)), \ldots, h(u(e, n-1))\rangle$$
	
	Now, by contradiction, suppose~\(g\) does not satisfy (2), \ie suppose there exists~\(e, n\) such that~\(|W_e^X|\leqslant n\) but~\(g(e, n)\in W_e^X\). Because~\(W_e^X\) has fewer than~\(n\) elements, we can suppose~\(g(e, n)\) is the~\(i^{\text{th}}\) one for a some~\(i<n\). Thus the function~\(\Phi_{u(e, i)}^X\) is constantly equal to~\(k_i\) where~\(g(e, n)=\langle k_0, \ldots, k_{n-1}\rangle\), in particular~\(\Phi_{u(e, i)}^X(u(e, i))=k_i\). But we also have
	$$g(e, n)=\langle f(u(e, 0)), \ldots, f(u(e, n-1))\rangle$$
	implying~\(f(u(e, i))=k_i=\Phi_{u(e, i)}^X(u(e, i))\), which is impossible as~\(f\) is supposed to be~\(\DNC_h(X)\).
\end{proof}

We are now ready to prove the following proposition. Conidis~\cite{conidistac} independently proved the same statement for~$\TAC$ with a similar construction. Note that by the equivalence of $\TAC+\BS^0_2$ with $\TCAC$, Conidis result implies \Cref{prop:tcac-dnch}. 

\begin{proposition}\label[proposition]{prop:tcac-dnch}
	Let~\(S \subseteq\NN^{<\NN}\) be an instance of~\(\SAC\). Every set~\(X\) of degree~\(\DNC_h(\emptyset')\), with~\(h\) a computable function, computes a solution of~\(S\).
\end{proposition}
\begin{remark}
	Once again, as in the case of \Cref{proba}, the proof here is carried purely as a computability statement, we have access to as much induction as we need. 
\end{remark}
\begin{proof}
	First, since~\(X\) is of degree~\(\DNC_h\) for a computable function~\(h\), by \Cref{dncequiv}, it computes a function~\(g:\NN^2\to\NN\) such that~\(\forall e, n, |W_e^{\emptyset'}|\leqslant n\implies g(e, n)\notin W_e^{\emptyset'}\) and which is dominated by a computable function~\(b:\NN^2\to\NN\).
	
	The idea of this proof is the same as in \Cref{proba}, but this time we are going to use~\(g\) to avoid selecting the potential ``bad" element in each finite antichain, \ie the element which is incompatible with only finitely many strings. For any finite set $A\subseteq S$, let $\psi_A : \NN \to S$ be a bijection such that \(\psi_A(\llbracket 0, |A|\llbracket)=A\).
	
	The procedure is the following. Initially, $S_0\coloneqq S$. At step~\(k\), assume $S_k \subseteq S$ has been defined. To find the desired antichain~\(A_k\) we use the fixed point theorem to find an index~\(e_k\) such that~\(\Phi_{e_k}^{\emptyset'}(n)\) is the procedure that halts if it finds
	an antichain~\(A\subseteq S_k\) whose size is greater than~\(b(e_k, 1)\) and~\(\psi_A(n)\in A\), and finds (using~\(\emptyset'\)) an integer \(m\) such that~\(\forall \ell>m, \psi_A(\ell)\succ\psi_A(n)\).
	
	Define~\(A_k\coloneqq A\). By choice of~$A$ and $e_k$,
	\[W_{e_k}^{\emptyset'}=\begin{cases}
		\{\psi^{-1}_A(\rho)\}&\text{if } A_k \text{ has a bad element }\rho\\
		\varnothing&\text{otherwise}
	\end{cases}\]

	Finally we can define~\(\sigma_k\coloneqq \psi_A(g(e_k, 1))\). Indeed since~\(|W_{e_k}^{\emptyset'}|\leqslant 1\) by construction, \(g(e_k, 1)\notin W_{e_k}^{\emptyset'}\). Moreover~\(\sigma_k\in A_k\), because~\(g(e_k, 1)<b(e_k, 1)<|A_k|\). This implies that~\(\sigma_k\) is not a bad element of~\(A_k\), in other words the set~\(S_{k+1}\coloneqq \{\tau\in S_k\mid \tau\bot \sigma_k \mbox{ and } |\tau| > |\sigma_k| \}\) is infinite.
\end{proof}

%

\section{\(\ADS\) and~\(\EM\)}\label[section]{sect:adsem}
Ramsey's theorem for pairs admits a famous decomposition into the Ascending Descending Sequence theorem ($\ADS$) and the Erd\H{o}s-Moser theorem ($\EM$) over~$\RCA_0$. As mentioned in the introduction, both statements are strictly weaker than~$\RT^2_2$. Actually, these statements are generally thought of as decomposing Ramsey's theorem for pairs into its disjunctiveness part with $\ADS$, and its compactness part with $\EM$. Indeed, the standard proof of $\ADS$ is disjunctive, and does not involve any notion of compactness, while the proof of $\EM$ is non-disjunctive and implies $\mathsf{RWKL}$, which is the compactness part of~$\RT^2_2$.

$\ADS$ and $\EM$ are relatively disjoint, in that they are only known to have the hyperimmunity principle as common consequence, which is a particularly weak principle. In this section however, we show that $\TCAC$ follows from both~\(\ADS\) and~\(\EM\) over~$\RCA_0$. We shall see in \Cref{sect:hyper} that $\TCAC$ implies the hyperimmunity principle.


The following proposition was proved by Dorais (personal communication) for $\SHER$ and independently by Conidis~\cite{conidistac} for $\TAC$. We adapted the proof of Dorais to obtain the following proposition. 

\begin{proposition}\label[proposition]{ads-impl-tcac}
	\(\RCA_0\vdash\ADS\implies\TCAC\) and~\(\TCAC\leqslant_c\ADS\)
\end{proposition}
\begin{proof}
	Let~\(T\subseteq \NN^{<\NN}\) be an infinite tree. Define the total order~\(<_0\) over~$T$ by
	\(\sigma<_0\tau \iff \sigma\prec\tau \lor (\sigma\bot\tau \land \sigma(d)<_\NN\tau(d))\)
	where~\(d\coloneqq \min\{k\in\NN\mid\sigma(k)\neq\tau(k)\}\). By~\(\ADS\), there is an infinite ascending or descending sequence~\((\sigma_i)\) for $(T, <_0)$. 
	
	If it is descending, there are two possibilities. Either~\(\forall^\infty i, \sigma_i\centernot\bot\sigma_{i+1}\), which means we eventually have an infinite~\(\prec\)-decreasing sequence of strings, which is impossible. Or~\(\exists^\infty i, \sigma_i\bot\sigma_{i+1}\), in which case we designate by~\((h_k)\) the sequence~\((\sigma_{\ell_k+1})_{k\in\NN}\) of all such~\(\sigma_{i+1}\), and we show that it is an antichain of~\(T\). 
	
	To do so, it suffices to prove by induction on~\(m\) that~\(\forall m>0, \forall k, h_k\centernot\succ h_{k+m}\). When~\(m=1\), due to how~\((\sigma_i)\) is structured, we have~\(\forall k, h_k\succcurlyeq \sigma_{\ell_k}\bot h_{k+1}\). We now consider~\(h_k\) and~\(h_{k+(m+1)}\). By induction hypothesis, \(h_k\centernot\succ h_{k+m}\) and~\(h_{k+m}\centernot\succ h_{k+(m+1)}\). Moreover since~\(h_k>_0 h_{k+m}>_0h_{k+(m+1)}\), we know there are minima~\(d\) and~\(e\) such that~\(h_k(d)>h_{k+m}(d)\) and~\(h_{k+m}(e)>h_{k+(m+1)}(e)\). Now~\(e\leqslant d\) implies~\(h_{k+(m+1)}(e)<h_{k+m}(e)=h_k(e)\), and~\(e > d\) implies~\(h_{k+(m+1)}(d)=h_{k+m}(d)<h_k(d)\); in any case~\(h_k\centernot\succ h_{k+(m+1)}\).
				
	Now if the sequence~\((\sigma_i)\) is ascending, we again distinguish two possibilities. Either~\(\forall^\infty i, \sigma_i\centernot\bot\sigma_{i+1}\), which means we eventually obtain an infinite path of the tree. Or~\(\exists^\infty i, \sigma_i\bot\sigma_{i+1}\), in which case we work in the same fashion as in the descending case (designate by~\((h_k)\) the sequence~\((\sigma_{\ell_k})_{k\in\NN}\) of all such~\(\sigma_i\), and show by induction on~\(m\) that~\(\forall m>0, \forall k, h_k\centernot\prec h_{k+m}\)). 
\end{proof}

\begin{proposition}
	\(\RCA_0\vdash\EM\implies\TCAC\) and~\(\TCAC\leqslant_c\EM\)
\end{proposition}
\begin{proof}
	Let~\(T\subseteq\NN^{<\NN}\) be an infinite tree. We first define a computable bijection~\(\psi:\NN\to T\). To do so, let~\(\varphi:\NN^{<\NN}\to\NN\) be the bijection~\( 
	x_0\cdot\ldots\cdot x_{n-1}\mapsto p_0^{x_0}\times\ldots\times p_{n-1}^{x_{n-1}}-1
	\)
	where~\(p_k\) is the~\(k^{\text{th}}\) prime number.
	The elements of the sequence~\((\varphi^{-1}(n))_{n\in\NN}\) that are in \(T\) form a subsequence denoted \((s_n)_{n\in\NN}\), and the function~\(\psi:\NN\to T\) is defined by \(n\mapsto s_n\).
	
	Note that, by construction, the range of $\psi$ is infinite and computable. Moreover, if $\sigma \prec \tau\in T$,
	then $\varphi(\sigma) < \varphi(\tau)$, hence $\psi^{-1}(\sigma) < \psi^{-1}(\tau)$.
	Also note that the range of~$\psi$ is not necessarily a tree.
	
	Let~\(f: [\NN]^2 \to 2\) be the coloring defined by $f(\{x, y\}) = 1$ iff $x <_\NN y$ and $\psi(x)\prec\psi(y)$ coincide. By~\(\EM\), there is an infinite transitive set~\(S\subseteq\NN\), \ie~\(\forall i<2, \forall x\lt y\lt z\in S, f(x, y)=f(y, z)=i\implies f(x, z)=i\)
	
	Note that if there are~\(x<y\in S\) such that~\(f(x, y)=0\), then~\(\forall z>y\in S, f(x, z)=0\). Indeed given~\(x<y<z\in S\) such that~\(f(x, y)=0\), either~\(f(y, z)=0\), and so by transitivity we have~\(f(x, z)=0\); or~\(f(y, z)=1\), but in that case~\(f(x, z)\neq 1\) because it is impossible to simultaneously have~\(\psi(y)\prec \psi(z)\), \(\psi(x)\prec \psi(z)\) and~\(\psi(x)\bot\psi(y)\).
	
	Now two cases are possible. Either~\(\exists^\infty j\in\NN, f(s_j, s_{j+1})=0\), so consider the infinite set~\(A\) made of all such~\(s_j\). Thanks to the previous property, \(A\) is~\(f\)-homogeneous for the color~\(0\), and so~\(\psi(A)\) is an infinite antichain.
	Or~\(\forall^\infty j\in\NN, f(s_j, s_{j+1})=1\), so there is a large enough~\(k\in\NN\) such that~\(\psi(s_k)\prec\psi(s_{k+1})\prec\ldots\), \ie we found an infinite path.
\end{proof}

\section{\(\TAC\), lowness and hyperimmunity}\label[section]{sect:hyper}
Binns et al.\ in \cite{binns2014} and Conidis~\cite{conidistac} respectively studied the reverse mathematics of $\TCAC$ and $\TAC$. Since~$\TCAC$ is computably equivalent to $\TAC+\BS^0_2$ and this equivalence also holds in reverse mathematics, the analysis of \(\TCAC\) and \(\TAC\) is very similar. For example, Binns et al.~\cite[Theorem 6.4]{binns2014} proved that for any fixed low set~$L$, there is a computable instance of~$\TCAC$ with no $L$-computable solution, while Conidis~\cite{conidistac} proved the existence of a computable instance of~$\TAC$ whose solutions are all of hyperimmune degree.
In this section, we prove a general statement regarding~\(\TAC\) (\Cref{nod2}) and show that it encompasses both results. 


\begin{theorem}\label[theorem]{nod2}
	Let~\((A_n)\) be an uniform sequence of infinite~\(\Delta_2^0\) sets. There is a computable instance of~\(\TAC\) such that no~\(A_n\) is a solution.
\end{theorem}
\begin{proof}
	First, for any~\(n\), let~\(e_n\) be the index of~\(A_n\), \ie~\(\Phi_{e_n}^{\emptyset'}=A_n\). We also write~\(A_n[s]\coloneqq \Phi_{e_n}^{\emptyset'[s]}[s]\).
	
	\textbf{Idea.}
	We are going to construct a tree~\(T\subseteq 2^{<\NN}\), such that
	for each~\(n\in\NN\), there is~\(\sigma_n\in A_n\) verifying~\(\sigma_n\notin T\) or~\(\sigma_n\in T\land \forall^\infty\tau\in T, \sigma_n\prec\tau\). These requirements are respectively denoted~\(\R_n\) and~\(\mathcal{S}_n\), and~\(A_n\) cannot be an infinite antichain of~\(T\) if one of them is met.
	
	The sequence~\((\sigma_n)\) is constructed via a movable marker procedure, with steps~\(s\) and sub-steps~\(e<s\). At each step~\(s\) we are going to manipulate an approximation~\(\sigma_n^s\) of~\(\sigma_n\), and variables~\(\widehat{\sigma}_n^s\) that will help us keep track of which requirement is satisfied by~\(\sigma_n^s\). 
	
	\textbf{Construction.}
	At the beginning of each step~\(s\), let~\(T_s\) be the approximation of the tree~\(T\) defined by~\(T_s\coloneqq T_{s-1}\cup\{\tau_s\cdot 0, \tau_s\cdot 1\}\) where~\(\tau_s\) is the leftmost (for example) leaf of~\(T_{s-1}\) such that~\(\tau_s\succcurlyeq \widehat{\sigma}_{s-1}^s\). For~\(s=0\), we let~\(T_0\coloneqq \{\varepsilon\}\).
	
	At step~\(s\), sub-step~\(e\), let
	\(\sigma_e^s\) be the string whose code is the smallest in the uniformly computable set~\(\{\tau\in A_e[s]\uh_{s}\mid (\tau\in T_s \land \tau\succcurlyeq\widehat{\sigma}_{e-1}^s) \lor \tau\notin T_s\}\)
	with~\(\widehat{\sigma}_{-1}^s \coloneqq  \varepsilon\)
	and~\(\sigma_e^s\) is undefined when the set is empty.
	
	Besides, define
	\(\widehat{\sigma}_e^s \coloneqq  \begin{cases}
		\sigma_e^s& \text{if }\sigma_e^s\in T_s\text{ (and therefore }\sigma_e^s\succcurlyeq\widehat{\sigma}_{e-1}^s)\\
		\widehat{\sigma}_{e-1}^s& \text{otherwise}
	\end{cases}\)
	
	\textbf{Verification.}
	By induction on~\(e\), we prove that~\(\sigma_e\coloneqq \lim_s\sigma_e^s\) exists and is an element of~\(A_e\), also we prove~\(\widehat{\sigma}_e\coloneqq \lim_s\widehat{\sigma}_e^s\) exists, and~\(\sigma_e\) satisfies~\(\R_e\) or~\(\mathcal{S}_e\).
	
	Suppose we reached a step~\(r\) such that for all~\(e'<e\) the values of~\(\sigma_{e'}^r\) and~\(\widehat{\sigma}_{e'}^r\) have stabilized. And thus, for any step \(s>r\), as \(\tau_s\succcurlyeq \widehat{\sigma}_{s-1}^s\succcurlyeq \widehat{\sigma}_{e-1}^s = \widehat{\sigma}_{e-1}\), the tree will always be extended with nodes above~\(\widehat{\sigma}_{e-1}\), implying only a finite part of the tree is not above~\(\widehat{\sigma}_{e-1}\). 
	
	Now suppose \(k\) is the smallest code of a string \(\tau\) such that \((\tau\in T\land \tau\succcurlyeq\widehat{\sigma}_{e-1})\lor \tau\notin T\). Such a string exists because \(A_e\) is infinite, whereas the set of strings in \(T\) that are below \(\widehat{\sigma}_{e-1}\) is not. If \(\tau\in T\), then \(\exists x, \forall y\geqslant x, \tau\in T_y\), otherwise define \(x\coloneqq 0\). Since \(A_e\) is~\(\Delta_2^0\), there exists \(s\geqslant \max\{k+1, r, x\}\) such that \(A_e[s]\uh_{k+1}\) has stabilized \ie~\(\forall t>s, A_e[t]\uh_{k+1}=A_e[s]\uh_{k+1}\). Thus \(\sigma_e^s=\tau\) because
	\(\tau\in A_e\uh_{k+1}= A_e[s]\uh_{k+1}\subseteq A_e[s]\uh_s\). This ensures that for any~\(t>s\), \(\sigma_e^t=\tau\), \ie~\(\sigma_e=\tau\).
	
	Finally, we distinguish two cases. Either~\(\sigma_e\in T\) and so~\(\exists t, \sigma_e^t\in T_t\), thus~\(\forall u>t, \widehat{\sigma}_e^u=\sigma_e^u\). So~\(\mathcal{S}_e\) is satisfied, as cofinitely many nodes of~\(T\) will be above~\(\widehat{\sigma}_e=\sigma_e\). Or~\(\sigma_e\notin T\), in which case, either~\(\forall t, \sigma_e^t\notin T_t\), implying~\(\forall t, \widehat{\sigma}_e^t\coloneqq \widehat{\sigma}_{e-1}^t\) and thus~\(\R_e\) is satisfied.
\end{proof}

We now show how \Cref{nod2} relates to the result of Binns et al.\ in~\cite[Theorem 6.4]{binns2014}, that is, the existence, for any fixed low set~$L$, of a computable instance of~$\TCAC$ with no $L$-computable solution.

\begin{lemma}\label[lemma]{d2low}
	For any low set~\(P\), the sequence of infinite~\(P\)-computable sets is uniformly~\(\Delta_2^0\).
\end{lemma}
\begin{proof}
	Since~\(P'\leqslant_T\emptyset'\) we can~\(\emptyset'\)-compute the function
	\[f(e, x)=\begin{cases}
		\Phi_e^P(x) &\text{ when }\forall y\leqslant x, \Phi_e^P(y)\downarrow\text{ and }\exists y>x, \Phi_e^P(y)\downarrow=1\\
		1 &\text{ otherwise}
	\end{cases}\]
	Now let \(A_e\coloneqq \{f(e, x)\mid x\in\NN\}\). If~\(\Phi_e^P\) is total and infinite then \(A_e\) is equal to it, so it is~\(P\)-computable.
	Otherwise~\(A_e\) is cofinite, and in particular it is infinite and~\(P\)-computable.
\end{proof}

We are now ready to state the result of Binns et al.\ in \cite[Theorem 6.4]{binns2014}, but for $\TAC$.

\begin{corollary}\label[corollary]{nolowsol}
	For any low set~\(P\), there exists a computable instance of~\(\TAC\) with no~\(P\)-computable solution.
\end{corollary}
\begin{proof}
	Given~\(P\), we can use \Cref{d2low} to obtain a uniform sequence, on which we apply \Cref{nod2}.
\end{proof}

The previous corollary is very useful to show that $\RCA_0 + \WKL \nvdash \TAC$, since there exists a model of~$\RCA_0 + \WKL$ below a low set. The following corollary will be useful to prove that the result of Binns et al.\ in \cite[Theorem 6.4]{binns2014} implies the result of Conidis.

\begin{corollary}\label[corollary]{tac-pa-nosolution}
	There exist a~\(\PA\) degree~\(P\) and an instance of~\(\TAC\) with no~\(P\)-computable solution.
\end{corollary}
\begin{proof}
	It follows from the existence of a low~\(\PA\) degree by the low basis theorem, see \cite[Corollary 2.2.]{jockusch1972ca}.
\end{proof}

The next proposition has two purposes. First, it will be used to show the existence of a computable instance of $\TAC$ whose solutions are all of hyperimmune degree (see \Cref{thm:tac-hyp}). Second, it shows that, for any such instance, one can choose two specific functionals to witness this hyperimmunity, without loss of generality (see \Cref{cor:tac-hyp-wlog}).

\begin{proposition}\label[proposition]{pahyper}
	 Let~\(T\) be an instance of~\(\TAC\). For any set~\(P\) of~\(\PA\) degree, if~\(T\) has no~\(P\)-computable solution, then for any solution~\((\sigma_n)_{n \in \NN}\), the function~\(t_{T,(\sigma_n)_{n \in \NN}} : n\mapsto \min\{t\mid \sigma_n\in T[t]\}\) or~\(\ell_{(\sigma_n)_{n \in \NN}}:n\mapsto|\sigma_n|\) is hyperimmune.
\end{proposition}
\begin{proof}
	By contraposition, suppose there exists a solution~\((\sigma_n)_{n \in \NN}\) such that~\(t_{T, (\sigma_n)_{n \in \NN}}\) and~\(\ell_{(\sigma_n)_{n \in \NN}}\) are computably dominated by~\(t\) and~\(\ell\) respectively.
	Then the set
	\[\left\{(\tau_n)_{n \in \NN} \;\middle| \begin{array}{l}
		(\tau_n)_{n \in \NN}\text{ is an infinite antichain of }T\\
		t_{T, (\tau_n)_{n \in \NN}}\leqslant t\text{ and }\ell_{(\tau_n)_{n \in \NN}}\leqslant\ell
	\end{array}\right\}
	\] is a non-empty~\(\Pi_1^0\) class. It is non-empty because~\((\sigma_n)_{n \in \NN}\) belongs to it, and to show it is a~\(\Pi_1^0\) class, it can be written as
		\[\left\{(\tau_n)_{n \in \NN}\;\middle|\ \forall n, \begin{array}{l}
			\forall m<n, \tau_n\bot\tau_m\\
			\tau_n\in T[t(n)]\\
			|\tau_n|\leqslant\ell(n)
		\end{array}\right\}\]
	
	Thanks to~\(\ell\), the number of elements at each level \(n\) of the tree associated to this class is computably bounded by \(2^{\ell(n)}\), thus it can be coded by a~\(\Pi_1^0\) class of~\(2^\NN\). Finally since~\(P\) is of~\(\PA\) degree, it computes an element of any~\(\Pi_1^0\) class of the Cantor space, hence the result.
\end{proof}

Combining \Cref{tac-pa-nosolution} and \Cref{pahyper}, we obtain the following theorem from Conidis~\cite{conidistac}.

\begin{theorem}[Conidis~\cite{conidistac}]\label[theorem]{thm:tac-hyp}
	There is a computable instance of~\(\TAC\) such that each solution is of hyperimmune degree.
\end{theorem}
\begin{proof}
	Let~\(P\) be of low PA degree. By using \Cref{nolowsol} we get a computable instance~\(T\) of~\(\TAC\) with no~\(P\)-computable solution. Thus, by using \Cref{pahyper} we deduce that, for any solution~\((\sigma_n)\), its function~\(t_{T, (\sigma_n)_{n \in \NN}}\) or~\(\ell_{(\sigma_n)_{n \in \NN}}\) is hyperimmune. And~\((\sigma_n)_{n \in \NN}\) computes both, since~\(T\) is computable ; meaning it is of hyperimmune degree.
\end{proof}

\begin{corollary}
	\(\RCA_0\vdash\TAC\implies\HYP\)
\end{corollary}

In his direct proof of \Cref{thm:tac-hyp}, Conidis~\cite{conidistac} constructed computable instance of~\(\TAC\) and two functionals $\Phi, \Psi$ such that for every solution $H$, either $\Phi^H$ or $\Psi^H$ is hyperimmune. Interestingly, \Cref{pahyper} can be used to show that $\Phi$ and $\Psi$ can be chosen to be \(t_{T, -}\) and \(\ell_-\), without loss of generality. 

\begin{corollary}\label[corollary]{cor:tac-hyp-wlog}
	For any instance~\(T\) of~\(\TAC\) whose solutions are all of hyperimmune degree, at least one of the function~\(t_{T, -}\) or~\(\ell_-\) is a witness.
\end{corollary}
\begin{proof}
	Let~\(T\) be an instance of~\(\TAC\) whose solutions are all of hyperimmune degree, and let~\((\sigma_n)_{n \in \NN}\) be such a solution. By contradiction, if we suppose~\(t_{T, (\sigma_n)_{n \in \NN}}\) and~\(\ell_{(\sigma_n)_{n \in \NN}}\) are both computably dominated, then by \Cref{pahyper}, \(T\) has a~\(P\)-computable solution. If we choose~\(P\) to be computably dominated, then it cannot compute a solution of hyperimmune degree, hence a contradiction.
\end{proof}

Note that for every (computable or not) instance of \(\TAC\), there is a solution~$(\sigma_n)_{n \in \NN}$ such that~\(\ell_{(\sigma_n)_{n \in \NN}}\) is dominated by the identity function, by picking any path, and building an antichain along it.

\section{\(\SHER\)}\label[section]{sect:sher}
We have seen in \Cref{sect:adsem} that $\TCAC$ follows from both $\ADS$ and~$\EM$ over~$\RCA_0$.
The proof of~$\TCAC$ from $\ADS$ used only one specific property of the partial order $(T, \prec)$, that we shall refer to as \emph{semi-heredity}. Dorais and al.~\cite{dorais2016uniform} introduced the principle ~$\SHER$, which is the restriction of Ramsey's theorem for pairs to semi-hereditary colorings. In this section, we show that the seemingly artificial principle $\SHER$ turns out to be equivalent to the rather natural principle~$\TCAC$. This equivalence can be seen as more step towards the robustness of~$\TCAC$.


\begin{definition}[Semi-heredity]
	A coloring~\(f:[\NN]^2\to 2\) is \textbf{semi-here\allowbreak ditary} for the color~\(i<2\) if
	$$\forall x\lt y\lt z, f(x,z)=f(y, z)=i \implies f(x, y)=i$$
\end{definition}

The name ``semi-heredity" comes from the contraposition of the previous definition~\(\forall x\lt y\lt z, f(x, y)=1-i\implies f(x, z)=1-i\lor f(y, z)=1-i\)

\begin{definition}[\(\SHER\), \cite{dorais2016uniform}]
	For any semi-hereditary coloring~\(f\), there exists an infinite~\(f\)-homogeneous set.
\end{definition}

The first proposition consists essentially of noticing that, given a set of strings~$T \subseteq \NN^{<\NN}$,
the partial order $(T, \prec)$ behaves like a semi-hereditary coloring. The whole technicality of the proposition comes from the definition of an injection $\psi : \NN \to T$ with some desired properties.

\begin{proposition}\label[proposition]{sher-impl-tcac}
	\(\RCA_0\vdash\SHER\implies\TCAC[\tce]\) and
	
	\(\TCAC[\tce]\leqslant_c\SHER\)
\end{proposition}
\begin{proof}
	Let~\(T\subseteq \NN^{<\NN}\) be an infinite \ce tree. 
	First, let~\(\varphi:\NN^{<\NN}\to\NN\) the bijection~\( 
	x_0\cdot\ldots\cdot x_{n-1}\mapsto p_0^{x_0}\times\ldots\times p_{n-1}^{x_{n-1}}-1
	\)
	where~\(p_k\) is the~\(k^{\text{th}}\) prime number.
	Define $\psi : \NN \to T$ by letting~$\psi(n)$ be the least $\sigma \in T$ (in order of apparition) such that $\phi(\sigma)$ is bigger than $\phi(\psi(0)), \phi(\psi(1)), \dots, \phi(\psi(n-1))$. 
	Note that, by construction, the range of $\psi$ is infinite and computable. Moreover, if $\sigma \prec \tau$,
	then $\varphi(\sigma) < \varphi(\tau)$, hence $\psi^{-1}(\sigma) < \psi^{-1}(\tau)$.
	Also note that the range of~$\psi$ is not necessarily a tree.

	Let~\(f: [\NN]^2 \to 2\) be the coloring defined by $f(\{x, y\}) = 1$ iff $x <_\NN y$ and $\psi(x)\prec\psi(y)$ coincide. Let us show that $f$ is semi-hereditary for the color~1.
	Suppose we have~\(x<y<z\) and that~\(f(x, z)=f(y, z)=1\), \ie letting~\(\sigma \coloneqq  \psi(x), \tau \coloneqq  \psi(y), \rho \coloneqq  \psi(z)\) then we have~\(\sigma\prec\rho\) and~\(\tau\prec\rho\), thus either~\(\sigma\prec\tau\) or~\(\tau\prec\sigma\). But since~\(x<y\), \ie~\(\psi^{-1}(\sigma)<\psi^{-1}(\tau)\), only~\(\sigma\prec\tau\) can hold due to the above note, meaning~\(f(x, y)=1\).

	By~\(\SHER\) applied to~\(f\), there is an infinite $f$-homogeneous set~\(H\). If it is homogeneous for the color 0, then the set~\(\psi(H)\) corresponds to an infinite antichain of~\(T\). Likewise, if it is homogeneous for the color 1, then the set~\(\psi(H)\) is an infinite path of~\(T\).
\end{proof}

We now prove the converse of the previous proposition.

\begin{definition}[Weak homogeneity]
	Given a coloring~\(f:[\NN]^2\to k\), a set~\(A\coloneqq \{a_0 < a_1 < \ldots\}\subseteq\NN\) is \textbf{weakly-homogeneous} for the color~\(i<k\) if~\(\forall j, f(a_j, a_{j+1})=i\)
\end{definition}

Before proving \Cref{tcac-impl-sher}, we need a technical lemma.

\begin{lemma}[\(\RCA_0\)]\label[lemma]{weaksher}
	Let~\(f:[\NN]^2\to 2\) be a semi-hereditary coloring for the color~\(i<2\).
	For every infinite set~\(A\coloneqq \{a_0 < a_1 <\ldots\}\) which is weakly-homogeneous for the color~\(i\), there is an infinite $f$-homogeneous subset~$B \subseteq A$.
\end{lemma}
\begin{proof}[Proof (Dorais)]
	We first show that any~\(a_j\) falls in one of these two categories:
	\begin{enumerate}
		\item~\(\forall k>j, f(a_j, a_k)=i\)
		\item~\(\exists \ell>j, 
			\big(\forall k\in\ \rrbracket j, \ell\llbracket, f(a_j, a_k)=i \ \land\ 
			\forall k\geqslant \ell, f(a_j, a_k)=1-i\big)\)
	\end{enumerate}
	Indeed, for any~\(\ell>j\) such that~\(f(a_j, a_\ell)=i\), by semi-heredity, \(f(a_j, a_{\ell-1})=i\). So with a finite induction we get~\(\forall k\in\  \rrbracket j, \ell\rrbracket, f(a_j, a_k)=i\).
	
	There are now two possibilities. Either there are finitely many~\(a_j\) of type 2, and so by removing these elements, the resulting set is \(f\)-homogeneous for the color~\(i\). Otherwise there are infinitely many~\(a_j\) of type 2, in which case one can define an infinite $f$-homogeneous subset for color~$1-i$ using~\(\BS_1^0(A)\) as follows: due to the observation above, ``\(a_j\) is of type 2" is equivalent to the \(\Sigma^0_1(A)\) formula \(\exists\ell>j, f(a_j, a_\ell)=1-i\). Thus, given a finite set of type 2 elements \(\{a_{j_0}, \ldots, a_{j_{k-1}}\}\), by \(\BS_1^0(A)\) there is \(b>\max\{j_0, \ldots, j_{k-1}\}\), and so \(j_k\) is defined as the  smallest integer \(j_k\) such that \(j_k\geqslant b\) and \(f(a_{j_{k-1}}, a_{j_k})=1-i\).

\end{proof}

\begin{proposition}\label[proposition]{tcac-impl-sher}
	\(\RCA_0\vdash\TCAC\implies\SHER\) and
	
	\(\SHER\leqslant_c\TCAC\)
\end{proposition}
\begin{proof}
	Let~\(f:[\NN]^2\to 2\) be a semi-hereditary coloring for the color~\(i<2\). We begin by constructing a tree~\(T\subseteq \NN^{<\NN}\) defined as~\(T\coloneqq \{\sigma_n\mid n\in\NN\}\), where~\(\sigma_n\) is the unique string which is:
	\begin{enumerate}
		\item strictly increasing (as a function), with last element~\(n\)
		\item weak-homogeneous for the color~\(i\), \ie~\(\forall k<|\sigma_n|-1, f(\sigma_n(k), \sigma_n(k+1))=i\)
		\item maximal as a weak-homogeneous set, \ie~\(\forall y < \sigma_n(0), f(y, \sigma_n(0))=1-i\) and
		\(\forall j<|\sigma_n|-1, \forall y\in\ \rrbracket\sigma_n(j), \sigma_n(j+1)\llbracket, f(\sigma_n(j), y)=1-i \lor f(y, \sigma_n(j+1))=1-i\)
	\end{enumerate}
	To ensure existence, unicity and that~\(T\) is a tree, we prove~\(\sigma_n\) is obtained via the following effective procedure. Start with the string~\(n\). If the string~\(s_0\cdot\ldots\cdot s_m\) has been constructed, then look for the biggest integer~\(j<s_0\) such that~\(f(j, s_0)=i\). If there is none, the process ends. Else, the process is repeated with the string~\(j\cdot s_0\cdot\ldots\cdot s_m\).
	
	The string obtained is maximal by construction. It is unique, because at each step, if there are two (or more) integers~\(j_0<j_1\) smaller than~\(s_0\) and such that~\(f(j_0, s_0)=f(j_1, s_0)=i\), then by semi-heredity we have~\(f(j_0, j_1)=1\). This means we will eventually add~\(j_0\) after~\(j_1\). In particular, the string contains all the~\(j<n\) such that~\(f(j, n)=i\).
	Moreover this shows that~\(T\) is a tree, since the procedure is the same at any point during construction.
	
	Now we can apply~\(\TCAC\) to~\(T\), leading to two possibilities. Either there is an infinite path, which is a weakly-homogeneous set for the color~\(i\) thanks to condition 2. And so apply \Cref{weaksher} to obtain a~\(f\)-homogeneous set for the color~\(i\).
	
	Or there is an infinite antichain, which is of the form~\((\sigma_{n_j})_{j\in\NN}\). Let us show the set~\(H\coloneqq \{n_j\mid j\in\NN\}\) is~\(f\)-homogeneous for the color~\(1-i\).
	Indeed, if $f(n_s, n_t) = i$ for some~$s < t$, then $n_s \in \sigma_{n_t}$, since $\sigma_{n_t}$ contains all the elements $y < n_t$ such that $f(y, n_t) = i$. But then $\sigma_{n_s} \prec \sigma_{n_t}$, contradicting the fact that \((\sigma_{n_j})_{j\in\NN}\) is an antichain.
\end{proof}

We end this section by studying~\(\RT_2^2\) with respect to 3-variables forbidden patterns.
As explained in the introduction, there are three basic 3-variables forbidden patterns, yielding the notions of semi-heredity, semi-ancestry and semi-transitivity, respectively. These forbidden patterns induce Ramsey-like statements of the form ``for any 2-coloring of pairs, there exists an infinite set which avoids some kind of forbidden patterns''. This statement applied to semi-transitivity yields a consequence of the Erd\H{o}s-Moser theorem, known to be strictly weaker than Ramsey's theorem for pairs over~$\RCA_0$. We now show that the two remaining forbidden patterns yield statements equivalent to~\(\RT_2^2\). This completes the picture of the reverse mathematics of Ramsey-like theorems for 3-variable forbidden patterns.

\begin{definition}[Semi-ancestry]
	A coloring~\(f:[\NN]^2\to 2\) has \textbf{semi-ancestry} for the color~\(i<2\) if
	$$\forall x\lt y\lt z, f(x,y)=f(x, z)=i \implies f(y, z)=i$$
\end{definition}

Before proving $\RT^2_2$ from the Ramsey-like statement about semi-ancestry over~$\RCA_0$, we need to prove that this statement implies~$\BSig_2$. This is done by proving the following principle.

\begin{definition}[\(\mathsf{D}_2^2\)]
	Every \(\Delta_2^0\) set admits an infinite subset in it or its complement.
\end{definition}

\begin{proposition}
	The statement ``for any 2-coloring of pairs, there exists an infinite set which has semi-ancestry for some color'' implies~\(\mathsf{D}_2^2\) over~\(\RCA_0\).
\end{proposition}
\begin{proof}
	Let~\(A\) be a~\(\Delta_2^0\) set whose approximations are~\((A_t)_{t\in\NN}\).
	We define the coloring~\(f(x, y)\coloneqq  \chi_{A_y}(x)\), and use the statement of the proposition to obtain an infinite set~\(B\) that has semi-ancestry for some color.
	
	If~\(B\) has semi-ancestry for the color 1, then~\(\forall x\lt y\lt z\in B, x\in A_y\land x\in A_z\implies y\in A_z\).
	Now either~\(B\subseteq\overline{A}\) and we are done.
	Or~\(\exists x\in A\cap B\), which means~\(\forall^\infty y\in B, x\in A_y\), implying that~\(\forall^\infty y>x\in B, \forall z>y\in B, y\in A_z\) by semi-ancestry, \ie~\(\forall^\infty y>x\in B, y\in A\). So we can compute a subset~\(H\) of~\(B\) which is infinite and such that~\(H\subseteq A\).
	
	This argument also works when~\(B\) has semi-ancestry for the color 0, we just need to switch~\(A\) and~\(\overline{A}\), as well as~\(\in\) and~\(\notin\), when needed.
\end{proof}

\begin{corollary}\label[corollary]{p1-impl-bsigma2}
	The statement ``for any 2-coloring of pairs, there exists an infinite set which has semi-ancestry for some color'' implies~\(\BS^0_2\) over~\(\RCA_0\).
\end{corollary}
\begin{proof}
	Immediate, since~\(\RCA_0\vdash\mathsf{D}_2^2\implies\BS^0_2\), see \cite[Theorem 1.4]{chong2010role}.
\end{proof}

\begin{proposition}\label[proposition]{p1-impl-rt22}
	The statement  ``for any 2-coloring of pairs, there exists an infinite set which has semi-ancestry for some color'' implies~\(\RT_2^2\) over~\(\RCA_0\) and over the computable reduction.
\end{proposition}
\begin{proof}
	Let~\(f:[\NN]^2\to 2\) be a coloring. We can apply the statement to obtain an infinite set~\(A\) which has semi-ancestry for the color~\(i\), \ie~\(\forall x\lt y\lt z\in A, f(x, y)=i\land f(x, z)=i\implies f(y, z)=i\). There are two possibilities. Either there exists~\(a\in A\) such that~\(\exists^\infty b>a\in A, f(a, b)=i\), in which case all such elements~\(b\) form an infinite~\(f\)-homogeneous set due to the property of~\(A\).
	Otherwise any~\(a\in A\) verifies~\(\forall^\infty b>a, f(a, b)=1-i\), \ie all the elements of~\(A\) have a limit color equal to~\(1-i\) for the coloring~\(f\uh_{[A]^2}\). Thus we can use~\(\BS^0_2\) (\Cref{p1-impl-bsigma2}) to compute an infinite homogeneous set (see \cite[Proposition 6.2]{dzhafarov2020reduction}).
\end{proof}

The proof that the Ramsey-like statement about semi-heredity implies Ramsey's theorem for pairs is indirect, and uses the Ascending Descending Sequence principle.

\begin{proposition}\label[proposition]{p2-impl-ads}
	The statement ``for any 2-coloring of pairs, there exists an infinite set which is semi-hereditary for some color'' implies~\(\ADS\) over~\(\RCA_0\) and over the computable reduction.
\end{proposition}
\begin{proof}
	Let~$\L = (\NN, <_\L)$ be a linear order. Let
	\(f: [\NN]^2 \to 2\) be the coloring defined by \(f(\{x, y\}) = 1\) iff $<_\L$ and $<_\NN$ coincide on $\{x, y\}$. By the statement of the proposition, there is an infinite set~\(H\) on which the coloring is semi-hereditary for some color~\(i\).
	
	Before continuing, note that if there is a pair~\(x<y\in H\) such that~\(f(x, y)=1-i\) then~\(\forall z>y\in H, f(x, z)=1-i\). Indeed, either~\(f(y, z)=1-i\), in which case by transivity of~\(<_\L\) we have~\(f(x, z)=1-i\). Or~\(f(y, z)=i\), in which case by semi-heredity~\(f(x, z)=1-i\), because otherwise it would imply that~\(f(x, y)=i\).
	
	Now if~\(\exists^\infty x\in H, \exists y>x\in H, f(x, y)=1-i\), then we construct an infinite decreasing sequence by induction. Suppose the sequence~\(x_0 >_\L \ldots >_\L  x_{n-1}\) has already be constructed, and we have~\(m\in H\) such that~\(f(x_{n-1}, m)=1-i\), then we look for a~\(\{x<y\}\) such that~\(x>m\) and~\(f(x, y)=1-i\). By the previous remark we have that~\(f(x_{n-1}, x)=1-i\), \ie~\(x_{n+1} >_\L x\), so we extend the sequence with~\(x\) and redefine~\(m\coloneqq y\).
	
	Else~\(\forall^\infty x\in H, \forall y>x\in H, f(x, y)=i\), and so getting rid of finitely many elements yields an infinite increasing sequence.
\end{proof}

\begin{corollary}\label[corollary]{p2-impl-rt22}
	The statement ``for any 2-coloring of pairs, there exists an infinite set which is semi-hereditary for some color'' implies~\(\RT_2^2\) over~\(\RCA_0\).
\end{corollary}
\begin{proof}
	This comes from the fact that~\(\RT_2^2\iff S+\SHER\) by definition (with~\(S\) denoting the statement in question). 
	And we have~\(\RCA_0\vdash S\implies\SHER\) by using \Cref{p2-impl-ads}, \Cref{ads-impl-tcac} and \Cref{tcac-impl-sher}.
\end{proof}

\begin{remark}\label[remark]{rk-question}
Let~$S$ denote the statement ``for any 2-coloring of pairs, there exists an infinite set which is semi-hereditary for some color''. The proof that $S$ implies~$\RT^2_2$ over~$\RCA_0$ involves two applications of~$S$. The first one to obtain an infinite set over which the coloring is semi-hereditary, and a second one to solve~$\SHER$ using the fact that~$S$ implies~$\ADS$, which itself implies~$\SHER$.
It is unknown whether~$\RT^2_2$ is computably reducible to~$S$.
\end{remark}

\section{Stable counterparts: \(\SADS\) and~\(\TCAC[stable\ \tce]\)}\label[section]{sect:stable}
Cholak, Jockusch and Slaman~\cite{cholak2001strength} made significant progress in the understanding of Ramsey's theorem for pairs by dividing the statement into a stable and a cohesive part. 

\begin{definition}
A coloring~$f : [\NN]^2 \to k$ is \textbf{stable} if for every~$x \in \NN$, $\lim_y f(x, y)$ exists.
A linear order $\L = (\NN, <_\L)$ is \textbf{stable} if it is of order type $\omega + \omega^{*}$.	
\end{definition}

We call $\SRT^2_k$ and $\SADS$ the restriction of $\RT^2_k$ and $\ADS$ to stable colorings and stable linear orders, respectively. Given a linear order $\L = (\NN, <_\L)$, the coloring corresponding to the order is stable iff the linear order is of order type $\omega+\omega^{*}$, or $\omega+k$ or $k+\omega^*$. In particular, $\SRT^2_2$ implies $\SADS$ over~$\RCA_0$.

In this section, we study the stable counterparts of $\TCAC$ and $\SHER$, and prove that they are equivalent over~$\RCA_0$. We show~\(\SADS\) implies~\(\TCAC[stable\ \tce]\) over~$\RCA_0$. It follows in particular that every computable instance of \(\TCAC[stable\ \tce]\) admits a low solution.

\begin{definition}[Stable tree, Dorais \cite{dorais2012hajnal}]
	A tree~\(T\subseteq \NN^{<\NN}\) is \textbf{stable} when for every~\(\sigma\in T\)
	either~\(\forall^\infty\tau\in T, \sigma\bot\tau\) or~\(\forall^\infty\tau\in T, \sigma\centernot\bot\tau\)
\end{definition}

Note that any stable finitely branching tree admits a unique path.

\begin{proposition}
	\(\RCA_0\vdash\SADS\implies\TCAC[stable\ \tce]\)
\end{proposition}
\begin{remark}
	In the proof below, we use the statement that~\(\RCA_0\vdash \SADS\implies\BS_2^0\). This is because~\(\RCA_0\vdash \SADS\implies\mathsf{PART}\) (\cite[Proposition 4.6]{hirschfeldt2007combinatorial}) and \(\RCA_0\vdash \mathsf{PART}\iff \BS_2^0\) (\cite[Theorem 1.2]{chong2010role}). 
\end{remark}
\begin{proof}
	Let~\(T\subseteq\NN^{<\NN}\) be an infinite \ce tree which is stable.
	
	Consider the total order~\(<_0\), defined by~\(\sigma<_0\tau \iff \sigma\prec\tau\lor(\sigma\bot\tau\land \sigma(d)<\tau(d))\) where~\(d\coloneqq \min\{y\mid \sigma(y)\neq\tau(y)\}\).
	We show that it is of type~\(\omega+\omega^*\), \ie
	\[\forall\sigma\in T, (\forall^\infty\tau\in T \sigma<_0\tau)\;\lor\;(\forall^\infty\tau\in T, \tau<_0\sigma)\]
	
	So let~\(\sigma\in T\), there are two possibilities. Either~\(\forall^\infty\tau\in T, \sigma\centernot\bot\tau\), meaning we even have~\(\forall^\infty\tau\in T, \sigma\prec\tau\), which directly implies~\(\forall^\infty\tau\in T, \sigma<_0\tau\).
	
	Or~\(\forall^\infty\tau\in T, \sigma\bot\tau\). In this case, we consider all the nodes~\(\tau\) successors of a prefix of~\(\sigma\) but not prefix of~\(\sigma\), there are finitely many of them, because there are finitely many prefixes of~\(\sigma\) and no infinitely-branching node (WLOG, as otherwise there would be a computable infinite antichain). So we can apply the pigeon-hole principle, by using~\(\BS_2^0\), to deduce that there is a certain~\(\tau\) which has infinitely many successors. 
	
	Moreover, by stability of~\(T\) there cannot be another such node. Indeed, by contradiction, if there were two such nodes~\(\tau\) and~\(\tau'\), then we would have that~\(\exists^\infty \eta\in T, \eta\bot\tau\), because the successors of~\(\tau'\) are incomparable to~\(\tau\). And since~\(\tau\) already verifies~\(\exists^\infty \eta\in T, \eta\succ\tau\), contradicting the stability of~\(T\).
	
	Therefore we have that~\(\forall^\infty \eta\in T, \eta\succ\tau\), and so depending on whether~\(\tau<_0\sigma\) or~\(\sigma<_0\tau\), we obtain that~\(\forall^\infty\eta\in T, \eta<_0\sigma\) or~\(\forall \eta\in T, \sigma<_0\eta\) respectively. From there we can apply~\(\SADS\) and the proof is exactly like in \Cref{ads-impl-tcac}.
\end{proof}

\begin{corollary}\label[corollary]{cor:tac-stable-low}
	\(\TCAC[stable\ \tce]\) admits low solutions.
\end{corollary}
\begin{proof}
	This comes from the fact that any instance of~\(\SADS\) has a low solution, as proven in \cite[Theorem 2.11]{hirschfeldt2007combinatorial}.
\end{proof}

The proof that $\SHER$ follows from $\TCAC[stable]$ will use $\BSig_2$.
We therefore first prove that $\TCAC[stable]$ implies $\BSig_2$ over~$\RCA_0$.

\begin{lemma}\label[lemma]{tcac-impl-bsigma}
	\(\RCA_0\vdash\TCAC[stable]\implies \forall k, \RT^1_k\)
\end{lemma}
\begin{proof}
	Let~\(f:\NN\to k\) be a coloring. There are two possibilities:
	Either~\(\exists i<k, \exists^\infty x, f(x)=i\), in which case there is an infinite computable~\(f\)-homogeneous set.
	Otherwise~\(\forall i<k, \forall^\infty x, f(x)\neq i\), in which case we consider the infinite tree
	\[T\coloneqq \{\sigma\in\mathsf{Inc} \mid \forall y\leqslant \max\sigma, f(y)=f(\max\sigma)\implies y\in\ran\sigma\}\]
	where \(\mathsf{Inc}\) is the set of strictly increasing strings of~$\NN^{<\NN}$.
	
	By hypothesis, for every node~$\sigma \in T$, \(\forall^\infty \tau\in T, \sigma\bot\tau\), otherwise $T$ would have an infinite $T$-computable path. Thus, $T$ is stable.
	Moreover, every antichain is of size at most~$k$, thus $T$ is a stable tree with no infinite path and no infinite antichain, contradicting~\(\TCAC[stable]\).
\end{proof}

\begin{proposition}
	Under \(\RCA_0\) the statement \(\TCAC[stable]\) implies \(\SHER\mathsf{\ for\ stable\ colorings}\)
\end{proposition}
\begin{proof}
	Let~\(f:[\NN]^2\to 2\) be a stable coloring, semi-hereditary for the color~\(i\). We distinguish two cases. Either there are finitely many integers with limit color~\(i\), meaning we can ignore them and use~\(\BS_2^0\) (\Cref{tcac-impl-bsigma}) to compute an infinite homogeneous set (see \cite[Proposition 6.2]{dzhafarov2020reduction}).
	Otherwise there are infinitely many integers whose limit color is~\(i\), in which case we use the same proof as in \Cref{tcac-impl-sher}, but we must prove the tree~\(T\) we construct is stable. So let~\(\sigma_n\) be a node of this tree. 
	
	Suppose first~\(n\) has limit color~\(i\). Let~$p$ be sufficiently large so that $f(n, p) = i$. As explained in \Cref{tcac-impl-sher}, $\sigma_p$ contains all the integers~$m < p$ such that $f(m, p) = i$. It follows that $n \in \sigma_p$. Moreover, if $n \in \sigma_p$, then $\sigma_n \preceq \sigma_p$. Thus,~\(\forall^\infty p, \sigma_n\prec\sigma_p\).
	
	Suppose now~\(n\) has limit color~\(1-i\), then since there are infinitely many integers with limit color~\(i\), there is one such that~\(p>n\). In particular,~\(\sigma_p\) verifies~\(\forall^\infty \tau\in T, \sigma_p\prec\tau\). Thus, if~\(\sigma_n\prec\sigma_p\) then~\(\forall^\infty \tau\in T, \sigma_n\prec\tau\), and if~\(\sigma_n\bot\sigma_p\) then~\(\forall^\infty \tau\in T, \sigma_n\bot\tau\).
\end{proof}

\begin{proposition} Under
	\(\RCA_0\) the statement \(\SHER\mathsf{\ for\ stable\ colorings}\) implies \(\TCAC[stable\ \tce]\)
\end{proposition}
\begin{proof}
	Let \(T\subseteq\NN^{<\NN}\) be an infinite stable \ce tree. The proof is the same as in \Cref{sher-impl-tcac}, but we must verify that the coloring \(f:[\NN]^2\to 2\) defined is stable. Given \(x\in\NN\), we claim that \(\exists i<2, \forall^\infty y f(x, y)=i\). Since \(T\) is stable, either \(\forall^\infty y, \psi(x)\centernot{\bot}\psi(y)\) or \(\forall^\infty y, \psi(x)\bot\psi(y)\) holds. In the first case \(\forall^\infty y, f(x, y)=1\), and in the second one \(\forall^\infty y, f(x, y)=0\). Thus the coloring is stable, and the proof can be carried on.
\end{proof}

\begin{corollary}
The following are equivalent over~$\RCA_0$:
\begin{enumerate}
	\item[(1)] $\TCAC[stable]$
	\item[(2)] $\TCAC[stable\ \tce]$
	\item[(3)] $\SHER\mathsf{\ for\ stable\ colorings}$
\end{enumerate}
\end{corollary}

\section{Conclusion}\label[section]{sect:conclu}

The following figure sums up the implications proved in this paper. All implications hold both in $\RCA_0$ and over the computable reduction.

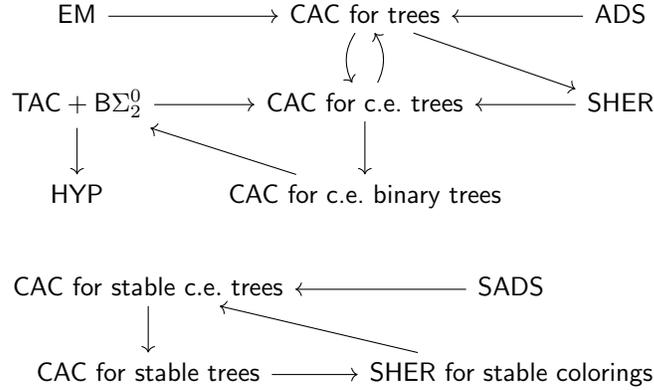
\begin{center}
	\begin{tikzcd}
		\textsf{EM} \arrow[r]            & \textsf{CAC for trees} \arrow[rd] \arrow[d, bend right]     & \textsf{ADS} \arrow[l]  \\
		\textsf{TAC} + \BS^0_2 \arrow[r] \arrow[d] & \textsf{CAC for \ce trees} \arrow[d] \arrow[u, bend right] & \textsf{SHER} \arrow[l] \\
		\textsf{HYP}                     & \textsf{CAC for \ce binary trees} \arrow[lu]               &                        
	\end{tikzcd}
\end{center}

\quad

\quad

\begin{center}
	\begin{tikzcd}
		{\TCAC[stable\ \tce]} \arrow[d] & \SADS \arrow[l]                                   \\
		{\TCAC[stable]} \arrow[r]       & \SHER\mathsf{\ for\ stable\ colorings} \arrow[lu]
	\end{tikzcd}
\end{center}

\quad

\quad

We have established in \Cref{big-equiv} the equivalence between \(\TAC+\BS^0_2\) and other statements.

\begin{question}
	What is the first-order part of \(\TAC\)?
\end{question}

Recall that, by \Cref{nolowsol}, for every fixed low set~$X$, there is a computable instance of~\(\TAC\) 
with no $X$-computable solution. By computable equivalence, this property also holds for~\(\TCAC\). 
It is however unknown whether \Cref{nolowsol} can be improved to defeat all low sets simultaneously.

\begin{question}
Does every computable instance of~\(\TCAC\) admit a low solution?
\end{question}

Note that by \Cref{cor:tac-stable-low}, any negative witness to the previous question would yield a non-stable tree.

We have also seen by \Cref{prop:tcac-dnch} that for every computable instance $T$ of~\(\TCAC\),
every computably bounded \(\DNC\) function relative to~$\emptyset'$ computes a solution to~$T$. The natural question would be whether we can improve this result to any~\(\DNC\) function relative to~$\emptyset'$.

\begin{question}
Is there some $X$ such that for every computable instance $T$ of~\(\TCAC\),
every \(\DNC\) function relative to~$X$ computes a solution to~$T$?
\end{question}

Note that in the case of a computable set~$X$, the answer is negative, as there exist~\(\DNC\) functions of low degree.

Finally, recall from \Cref{rk-question} the following question:
\begin{question}
	Is \(\RT_2^2\) computably reducible to the statement ``for any 2-coloring of pairs, there exists an infinite set which is semi-hereditary for some color"?
\end{question}

\bibliographystyle{asl.bst}
\bibliography{bibliography.bib}

\end{document}

%% file: macros.tex
\newcommand{\uh}[0]{{\upharpoonright}}

\renewcommand{\AA}{\mathbb{A}}

\newcommand{\NN}{\mathbb{N}}

\renewcommand{\L}{\mathcal{L}}

\newcommand{\R}{\mathcal{R}}

\newcommand{\U}{\mathcal{U}}

\renewcommand{\O}{\mathcal{O}}

\newcommand{\RCA}{\mathsf{RCA}}

\newcommand{\WKL}{\mathsf{WKL}}

\newcommand{\RT}{\mathsf{RT}}
\newcommand{\COH}{\mathsf{COH}}
\newcommand{\SRT}{\mathsf{SRT}}
\newcommand{\ADS}{\mathsf{ADS}}
\newcommand{\EM}{\mathsf{EM}}

\newcommand{\BSig}{\mathsf{B}\Sigma^0}
\newcommand{\BPi}{\mathsf{B}\Pi^0}